\documentclass[10pt]{article}
\usepackage[utf8]{inputenc}
\usepackage{lmodern}
\usepackage[T1]{fontenc}
\usepackage{amsmath}
\usepackage{amsfonts}
\usepackage{amssymb}
\usepackage{amsthm}
\usepackage{mathrsfs} 
\usepackage{stmaryrd}
\usepackage{graphicx}
\usepackage[left=2.9cm,right=2.9cm,top=3cm,bottom=3cm]{geometry}
\usepackage{enumitem}
\usepackage{tikz}
\usepackage{dsfont}
\usepackage{multicol}
\usepackage{xcolor}

\usepackage[maxbibnames=99]{biblatex}
\renewcommand{\epsilon}{\varepsilon}
\DeclareMathOperator{\supp}{supp}

\theoremstyle{plain}
\newtheorem{thm}{Theorem}
\newtheorem{lem}[thm]{Lemma}
\newtheorem{prop}[thm]{Proposition}
\newtheorem{cor}[thm]{Corollary}

\theoremstyle{definition}
\newtheorem{defn}[thm]{Definition}

\theoremstyle{remark}
\newtheorem{rem}[thm]{Remark}

\begin{document}

\begin{center}
\large{\textbf{\uppercase{The damped focusing cubic wave equation on a bounded domain}}}

\vspace{\baselineskip}

\large{\textsc{Thomas Perrin}}

\vspace{\baselineskip}

\large{\today}

\vspace{\baselineskip}
\end{center}

\noindent
\textbf{Abstract.} For the focusing cubic wave (or Klein-Gordon) equation on a compact Riemannian manifold of dimension $3$, the dichotomy between global existence and blow-up for solutions starting below the energy of the ground state is known since the work of Payne and Sattinger. In the case of a damped equation, we prove that the dichotomy between global existence and blow-up still holds. In particular, the damping does not prevent blow-up. Assuming that the damping satisfies the geometric control condition, we then prove that any global solution converges to a stationary solution along a time sequence, and that global solutions below the energy of the ground state can be stabilised, adapting the proof of a similar result in the defocusing case. 

\section*{Introduction}

Let $\Omega$ and $\partial \Omega$ be the interior and the boundary of a smooth compact connected Riemannian manifold of dimension $3$. Let $\beta \in \mathbb{R}$ be such that the Poincaré inequality
\[\int_\Omega \left(\left\vert \nabla u \right\vert^2 + \beta \vert u \vert^2 \right) \mathrm{d}x \gtrsim \int_\Omega \vert u \vert^2 \mathrm{d}x\]
is satisfied, for all $u \in H_0^1(\Omega)$. This specifically requires $\beta > 0$ if $\partial \Omega$ is empty. For real-valued initial data $\left( u^0, u^1 \right) \in H_0^1(\Omega) \times L^2(\Omega)$, consider the equation
\[\left \{
\begin{array}{rcccl}\label{KG_nL}\tag{$\ast$}
\square u + \beta u & = & u^3 & \quad & \text{in } \mathbb{R} \times \Omega, \\
(u(0), \partial_t u(0)) & = & \left( u^0, u^1 \right) & \quad & \text{in } \Omega, \\
u & = & 0 & \quad & \text{on } \mathbb{R} \times \partial \Omega.
\end{array}
\right.\]
If $\beta = 0$, then (\ref{KG_nL}) is the focusing cubic wave equation, and if $\beta > 0$, it is the focusing cubic Klein-Gordon equation. For $\left( u^0, u^1 \right) \in H_0^1(\Omega) \times L^2(\Omega)$, write $\left\Vert u^0 \right\Vert_{H_0^1}^2 = \int_\Omega \left(\left\vert \nabla u^0 \right\vert^2 + \beta \vert u^0 \vert^2 \right) \mathrm{d}x$ and set 
\[J(u^0) = \frac{1}{2} \left\Vert u^0 \right\Vert_{H_0^1}^2 - \frac{1}{4} \left\Vert u^0 \right\Vert_{L^4}^4 \quad \text{ and } \quad E\left( u^0, u^1 \right) = J(u^0) + \frac{1}{2} \left\Vert u^1 \right\Vert_{L^2}^2.\]
The functionals $J$ and $E$ are respectively called \emph{static energy} and \emph{energy}. The energy of a solution of (\ref{KG_nL}) at time $t \in \mathbb{R}$ is defined by $E\left( u(t), \partial_t u(t) \right)$, and is conserved. Let $Q$ be a ground state of (\ref{KG_nL}), that is, a positive stationary solution of (\ref{KG_nL}) of minimal energy 
\[E(Q, 0) = J(Q) = d > 0.\]
See Theorem \ref{thm_existence_uniqueness_Q} for a precise definition. For $u^0 \in H_0^1(\Omega)$, write 
\begin{equation}\label{eq_def_K}
K(u^0) = \left\Vert u^0 \right\Vert_{H_0^1}^2 - \left\Vert u^0 \right\Vert_{L^4}^4,
\end{equation}
and set
\[\left \{
\begin{array}{c}
\mathscr{K}^+ = \left\{ \left( u^0, u^1 \right) \in H_0^1(\Omega) \times L^2(\Omega), E\left(u^0, u^1 \right) < d, K(u^0) \geq 0 \right\}, \\
\mathscr{K}^- = \left\{ \left( u^0, u^1 \right) \in H_0^1(\Omega) \times L^2(\Omega), E\left(u^0, u^1 \right) < d, K(u^0) < 0 \right\}.
\end{array}
\right.\]
The following result is due to Payne and Sattinger \cite{Payne-Sattinger}.

\begin{thm}\label{thm_dichotomy_introduction}
The spaces $\mathscr{K}^+$ and $\mathscr{K}^-$ are stable under the flow of \textnormal{(\ref{KG_nL})}. A solution starting from $\mathscr{K}^+$ is defined on $\mathbb{R}$, and a solution starting from $\mathscr{K}^-$ blows up in finite positive and negative times.
\end{thm}

For $\gamma \in L^\infty(\Omega, \mathbb{R}_+)$, our main focus is the damped equation
\[\left \{
\begin{array}{rcccl}\label{KG_nL_damped}\tag{$\ast \ast$}
\square u + \gamma \partial_t u + \beta u & = & u^3 & \quad & \text{in } \mathbb{R}_+ \times \Omega, \\
(u(0), \partial_t u(0)) & = & \left( u^0, u^1 \right) & \quad & \text{in } \Omega, \\
u & = & 0 & \quad & \text{on } \mathbb{R}_+ \times \partial \Omega,
\end{array}
\right.\]
with $\left( u^0, u^1 \right) \in H_0^1(\Omega) \times L^2(\Omega)$. For $t \geq 0$, the energy equality is
\[E\left( u(t), \partial_t u(t) \right) = E\left( u^0, u^1 \right) - \int_0^t \int_\Omega \gamma(x) \left\vert \partial_t u(s, x) \right\vert^2 \mathrm{d}x \mathrm{d}s.\]
In particular, the energy of a solution is nonincreasing. Our first result is the following theorem.

\begin{thm}\label{thm_dichotomy_damped_introduction}
Fix $\gamma \in L^\infty(\Omega, \mathbb{R}_+)$. The spaces $\mathscr{K}^+$ and $\mathscr{K}^-$ are stable under the forward flow of \textnormal{(\ref{KG_nL_damped})}. A solution of \textnormal{(\ref{KG_nL_damped})} initiated in $\mathscr{K}^+$ is defined on $\mathbb{R}_+$, and a solution of \textnormal{(\ref{KG_nL_damped})} initiated in $\mathscr{K}^-$ blows up in finite time $t > 0$. 
\end{thm}

In particular, below the energy of the ground state, a blow-up solution of (\ref{KG_nL}) cannot be stabilised by the addition of a damping term of the form $\gamma \partial_t u$. 

We will use the notion of generalized geodesic, for which we refer to \cite{MelroseSjostrand}. In this article, we always assume that no generalized geodesic has a contact of infinite order with $\partial \Omega$ (see \cite{BLR} for some details about this assumption).

\begin{defn}
For $\omega \subset \Omega$, we say that $\omega$ satisfies the \emph{Geometric Control Condition} (in short, GCC) if there exists $L > 0$ such that any generalized geodesic of $\Omega$ of length $L$ meets the set $\omega$.
\end{defn}

In the defocusing case, that is, with the non-linearity $u^3$ changed into $-u^3$, the energy of a solution is positive and governs its size. All corresponding solutions are defined on $\mathbb{R}$, and the only stationary solution is zero. In this case, Joly and Laurent \cite{Joly-Laurent} proved a stabilization result under the GCC, on bounded or unbounded domains. In \cite{JolyLaurentBis}, the same authors consider the case of a non-linearity that is asymptotically defocusing and establish a semi-global controllability result. For this type of non-linearity, non-zero stationary solutions exist, but no blow-up phenomenon occurs.

In the focusing case we consider here, few stabilization or controllability results exist in the literature. For a focusing equation on the entire space, or outside a star-shaped obstacle, Alaoui, Ibrahim, and Nakanishi \cite{Alaoui} established stabilization for solutions starting from $\mathcal{K}^+$ when the damping is greater than a positive constant outside a ball. Note that assuming that the obstacle is star-shaped is stronger than the GCC. Burq, Raugel, and Schlag \cite{BurqRaugelSchlag} proved a dichotomy property for radial solutions of the focusing damped Klein-Gordon equation on the entire space with a constant positive damping: such a solution either blows up in a finite positive time or exists globally in positive time and converges to a stationary solution. In \cite{Burq_Raugel_Schlag_weak}, the same authors established a similar dichotomy property, again in the radial case, with a damping independent of $x$ and converging to zero as $t$ tends to infinity. In the case of the cubic focusing radial Klein-Gordon equation on a closed ball of $\mathbb{R}^3$, Krieger and Xiang \cite{KriegerXiang} studied the stability of the system near the static solution $u = 1$, under different dissipative boundary conditions. 

Our second result is the following theorem, which can be seen as an extension of the stabilisation property proved in \cite{Joly-Laurent} to the focusing case.

\begin{thm}\label{thm_stab_GCC_introduction}
Suppose $\gamma(x) \geq \alpha$ for almost all $x \in \omega$, with $\alpha > 0$ a constant and $\omega \subset \Omega$ an open set fulfilling the GCC. Then for any $E_0 \in [0, d)$, there exist $C > 0$ and $\lambda > 0$ such that for all $\left( u^0, u^1 \right) \in \mathscr{K}^+$, the solution $u$ of \textnormal{(\ref{KG_nL_damped})} satisfies
\[\left\Vert u(t) \right\Vert_{H_0^1} + \left\Vert \partial_t u(t) \right\Vert_{L^2} \leq C e^{-\lambda t}, \quad t \geq 0,\]
if $E\left( u^0, u^1 \right) \leq E_0$.
\end{thm}

For a solution starting from $\mathscr{K}^+$, we will see that the energy can be used almost as in the defocusing case, allowing us to adapt most of the arguments of \cite{Joly-Laurent}. Note that $C$ and $\lambda$ depend on $E_0$ (see Remark \ref{rem_constant_dependance}).

The proof of Theorem \ref{thm_stab_GCC_introduction} will rely on an asymptotic compactness result, which roughly states that for all uniformly bounded sequence of global solutions of (\ref{KG_nL_damped}), one can find a subsequence which converges to a stationary solution of (\ref{KG_nL}) along a time-sequence; see Proposition \ref{prop_cv_sequence_time} for a precise statement. It implies the following theorem.

\begin{thm}\label{thm_cv_global_solution}
Fix $\gamma \in L^\infty(\Omega, \mathbb{R}_+)$ satisfying the assumption of Theorem \ref{thm_stab_GCC_introduction}. Let $\left( u^0, u^1 \right) \in H_0^1(\Omega) \times L^2(\Omega)$ be such that the associated solution $u$ of \textnormal{(\ref{KG_nL_damped})} exists on $\mathbb{R}_+$. 
\begin{enumerate}
\item For all sequence $(T_n)_{n \in \mathbb{N}}$ such that $T_n \rightarrow + \infty$, there exist an increasing function $\phi: \mathbb{N} \rightarrow \mathbb{N}$ and a stationary solution $w$ of \textnormal{(\ref{KG_nL})} such that
\[\left( u(T_{\phi(n)} + \cdot), \partial_t u(T_{\phi(n)} + \cdot) \right) \xrightarrow{n \rightarrow \infty} (w, 0)\]
in $L^\infty_\mathrm{loc}(\mathbb{R}, H_0^1(\Omega) \times L^2(\Omega))$.
\item Assume that there exists an at most countable number of stationary solution $w$ of \textnormal{(\ref{KG_nL})} such that
\begin{equation}\label{eq_thm_cv_bounded_solution}
J(w) = \liminf_{t \rightarrow + \infty} E\left( u(t), \partial_t u(t) \right).
\end{equation}
Then there exists a stationary solution $w$ of \textnormal{(\ref{KG_nL})} such that 
\[\left( u(t + \cdot), \partial_t u(t + \cdot) \right) \xrightarrow{t \rightarrow + \infty} (w, 0)\]
in $L^\infty_\mathrm{loc}(\mathbb{R}, H_0^1(\Omega) \times L^2(\Omega))$.
\end{enumerate}
\end{thm}

\begin{rem}
For example, if $\Omega$ is an open subset of $\mathbb{R}^3$ and $\liminf E\left( u(t), \partial_t u(t) \right) \leq J(Q)$, then there are at most 3 stationary solutions satisfying (\ref{eq_thm_cv_bounded_solution}), which are $0$, $Q$ and $-Q$ (see Theorem \ref{thm_existence_uniqueness_Q}).
\end{rem}

\begin{rem}
In a companion paper \cite{perrinUniform}, we prove that all solutions of \textnormal{(\ref{KG_nL_damped})} defined on $\mathbb{R}_+$ are bounded in the energy space $H_0^1(\Omega) \times L^2(\Omega)$, generalizing a result of Cazenave \cite{Cazenave}. Our proof of Theorem \ref{thm_cv_global_solution} relies on this fact.
\end{rem}

On the one hand, observe that Theorem \ref{thm_cv_global_solution} applies to any solution defined on $\mathbb{R}_+$, while Theorem \ref{thm_stab_GCC_introduction} only applies to solutions with energy below the energy of the ground state. On the other hand, Theorem \ref{thm_stab_GCC_introduction} provides an exponential convergence rate, whereas Theorem \ref{thm_cv_global_solution} provides no convergence rate.

In \cite{BurqRaugelSchlag}, the authors proved that any radial solution defined on $\mathbb{R}_+$ converges to an asymptotic solution, when considering $\Omega = \mathbb{R}^d$. They make no assumption regarding the number of stationary solutions at a fixed energy level. Their damping is a positive constant, whereas Theorem \ref{thm_cv_global_solution} holds true for a damping satisfying the GCC. The fact that a solution defined on $\mathbb{R}_+$ is bounded in the energy space arises as a consequence of their main result, using dynamical systems arguments.


\paragraph{Outline of the article.} In Section 1, we recall some basic facts about (\ref{KG_nL_damped}) and the stabilisation of the linear version of (\ref{KG_nL_damped}), and we prove some technical results related to the ground state $Q$. In Section 2, we prove Theorem \ref{thm_dichotomy_damped_introduction}. In Section 3, we establish the asymptotic compactness result (Proposition \ref{prop_cv_sequence_time}), and we show that it implies Theorem \ref{thm_cv_global_solution}. In Section 4, we first prove an easy stabilisation property in the case of a positive damping, and then we prove Theorem \ref{thm_stab_GCC_introduction}, by adapting the proof of \cite{Joly-Laurent}.

\paragraph{Acknowledgements.} I warmly thank Thomas Duyckaerts and Jérôme Le Rousseau for their constant support and guidance. 

\vspace{\baselineskip}
\noindent
\textit{Keywords:} damped wave equations, nonlinear wave equation, focusing nonlinearity, exponential stabilization, blow-up, convergence to equilibrium.
\newline
\textit{MSC2020:} 35B40, 35B44, 35L71, 93D20.







\section{Preliminaries}

\subsection{The Cauchy problem and some properties of the linear equation}

In this article, we only consider real-valued solutions of wave equations. For $s \in [0, 1]$, let $H_0^s(\Omega)$ denote the complex interpolation space between $L^2(\Omega)$ and $H_0^1(\Omega)$. Define
\begin{equation}\label{eq_def_X^s}
X^s = \left( H^{1 + s}(\Omega) \cap H_0^1(\Omega) \right) \times H_0^s(\Omega).
\end{equation}
For example, one has $X^0 = H_0^1(\Omega) \times L^2(\Omega)$ and $X^1 = \left( H^2(\Omega) \cap H_0^1(\Omega) \right) \times H_0^1(\Omega)$. For $s \geq 0$, write $D\left( \Delta^s \right)$ for the iterated domain of the Dirichlet Laplacian (defined by interpolation for $s \notin \mathbb{N}$). Then for $s\in [0, 1]$, one has
\begin{equation}\label{eq_X^s_domains_Delta}
X^s = D\left( \Delta^{\frac{1 + s}{2}} \right) \times D\left( \Delta^{\frac{s}{2}} \right)
\end{equation}
so the norm of $X^s$ can be represented by a Fourier series. Consider the unbounded operator $A: X^0 \rightarrow X^0$ with domain $D(A) = X^1$ defined by
\[A = \left(
\begin{array}{cc}
0 & \mathrm{Id} \\ 
\Delta - \beta & - \gamma
\end{array}
\right).\]
Write $e^{tA}$ for the associated semi-group. The following result about the Cauchy problem corresponding to the linear version of (\ref{KG_nL_damped}) is well-known.

\begin{thm}\label{thm_existence_linear_waves}
For $s \in [0, 1]$, $\left( u^0, u^1 \right) \in X^s$, $\gamma \in L^\infty(\Omega, \mathbb{R}_+)$, and $g \in L_\mathrm{loc}^1(\mathbb{R}, H_0^s(\Omega))$, there exists a unique solution 
\[u \in \mathscr{C}^0(\mathbb{R}, H^{1 + s}(\Omega) \cap H_0^1(\Omega)) \ \cap \ \mathscr{C}^1(\mathbb{R}, H_0^s(\Omega))\]
of the linear damped wave equation 
\[\left \{
\begin{array}{rcccl}\label{KG_L_damped}\tag{$\ast \ast \ast$}
\square u + \gamma \partial_t u + \beta u & = & g & \quad & \text{in } \mathbb{R} \times \Omega, \\
(u(0), \partial_t u(0)) & = & \left( u^0, u^1 \right) & \quad & \text{in } \Omega, \\
u & = & 0 & \quad & \text{on } \mathbb{R} \times \partial \Omega.
\end{array}
\right.\]
In addition, for $T > 0$ and $s \in [0, 1]$, there exists $C_{s, T} > 0$ such that for all $\left( u^0, u^1 \right) \in X^s$, $\gamma \in L^\infty(\Omega, \mathbb{R}_+)$, and $g \in L_\mathrm{loc}^1(\mathbb{R}, H_0^s(\Omega))$, one has
\[\left\Vert u \right\Vert_{L^\infty((0, T), H^{1 + s} \cap H_0^1)} + \left\Vert \partial_t u \right\Vert_{L^\infty((0, T), H_0^s)} \leq C_{s, T} \left( \left\Vert \left( u^0, u^1 \right) \right\Vert_{X^s} + \left\Vert g \right\Vert_{L^1((0, T), H_0^s)} \right).\]
\end{thm}

\begin{rem}\label{rem_constant_linear_estimate}
By the energy estimate, we can assume that $C_{0, T} \lesssim C_{0, 1}$ for $0 < T \leq 1$.
\end{rem}

To construct the solution of (\ref{KG_nL_damped}), one uses Strichartz estimates (see for example Theorem 2.1 of \cite{Joly-Laurent}).

\begin{thm}[Strichartz estimates]\label{thm_strichartz_estimates}
Let $T > 0$. There exists a constant $C > 0$ such that for all $(q, r)$ satisfying
\begin{equation}\label{admissible_exponents}
\frac{1}{q} + \frac{3}{r} = \frac{1}{2} \quad \text{and} \quad q \in \left[ \frac{7}{2}, +\infty \right]
\end{equation}
and for all $\left( u^0, u^1 \right) \in X^0$, $g \in L^1([0, T], L^2(\Omega))$, the unique solution $u$ of \textnormal{(\ref{KG_L_damped})} associated with $\left( u^0, u^1 \right)$ and $g$ satisfies
\[\Vert u \Vert_{L^q([0, T], L^r)} \leq C \left( \Vert u^0 \Vert_{H_0^1} + \Vert u^1 \Vert_{L^2} + \Vert g \Vert_{L^1([0, T], L^2)} \right).\]
\end{thm}

Using those estimates, the solution of the Cauchy problem (\ref{KG_nL_damped}) can be constructed locally (see \cite{GinibreVelo}). Unlike the defocusing case, it cannot be proven that the solution is global in time.

\begin{thm}\label{thm_existence_damped_waves}
Consider $\gamma \in L^\infty(\Omega, \mathbb{R}_+)$. For any (real-valued) initial data $\left( u^0, u^1 \right) \in X^0$, there exist a maximal time of existence $T \in (0, + \infty]$ and a unique solution $u$ of \textnormal{(\ref{KG_nL_damped})} in $\mathscr{C}^0([0, T), H_0^1(\Omega)) \cap \mathscr{C}^1([0, T), L^2(\Omega))$. If $T < + \infty$, then 
\begin{equation}\label{eq_thm_existence_wave_blow_up}
\left\Vert u(t) \right\Vert_{H_0^1} \xrightarrow{t \rightarrow T^-} + \infty.
\end{equation}
\end{thm}

\begin{proof}
We only prove (\ref{eq_thm_existence_wave_blow_up}). If $T < + \infty$, then the Cauchy theory gives 
\[\left\Vert u(t) \right\Vert_{H_0^1} + \left\Vert \partial_t u(t) \right\Vert_{L^2} \xrightarrow{t \rightarrow T^-} + \infty.\]
As the energy of a solution is nonincreasing, it implies 
\[\left\Vert u(t) \right\Vert_{L^4} \xrightarrow{t \rightarrow T^-} + \infty,\]
\[\left\Vert u(t) \right\Vert_{L^4} \xrightarrow{t \rightarrow T^-} + \infty,\]
yielding (\ref{eq_thm_existence_wave_blow_up}) by the Sobolev embedding $H^1(\Omega) \hookrightarrow L^4(\Omega)$.
\end{proof}

In what follows, we refer to $u$ as "the solution of \textnormal{(\ref{KG_nL_damped})} with initial data $\left( u^0, u^1 \right)$ (and damping $\gamma$)". Finally, we will need the fact that the GCC implies the stabilisation of the linear equation. More precisely, we will use the following result, which is due to Bardos, Lebeau and Rauch \cite{BLR}.

\begin{thm}\label{thm_stab_linear}
Assume that $\gamma \in L^\infty(\Omega, \mathbb{R}_+)$ satisfies $\gamma(x) \geq \alpha$ for almost all $x \in \omega$, with $\alpha > 0$ a constant and $\omega \subset \Omega$ an open set fulfilling the GCC. There exist $C > 0$ and $\lambda > 0$ such that for all $s \in [0, 1]$, and all $\left( u^0, u^1 \right) \in X^s$, if $u = e^{tA} \left( u^0, u^1 \right)$ is the solution of \textnormal{(\ref{KG_L_damped})} with $g = 0$, then
\[\left\Vert \left( u(t), \partial_t u(t) \right) \right\Vert_{X^s} \leq C e^{- \lambda t} \left\Vert \left( u^0, u^1 \right) \right\Vert_{X^s}, \quad t \geq 0.\]
\end{thm}

\subsection{Ground state and some related properties}

Consider $u \in H_0^1(\Omega)$, $u \neq 0$. For $\lambda \in \mathbb{R}$, set $j(\lambda) = J(\lambda u)$, that is
\[j(\lambda) = \frac{\lambda^2}{2} \left\Vert u \right\Vert^2_{H_0^1} - \frac{\lambda^4}{4} \left\Vert u \right\Vert^4_{L^4}.\]
Write $\lambda^\ast = \lambda^\ast(u) > 0$ for the positive argument of the maximum of $j$. Using $j^\prime(\lambda^\ast) = 0$, one finds $K\left( \lambda^\ast u \right) = 0$, where $K$ is defined by (\ref{eq_def_K}). Set
\begin{align}
d & = \inf\left\{ J\left( \lambda^\ast(u) u \right), u \in H_0^1(\Omega), u \neq 0 \right\} \nonumber\\
& = \inf\left\{ J(w), w \in H_0^1(\Omega), w \neq 0, K(w) = 0 \right\} \label{eq_def_potential_well}.
\end{align}
The fact that equality (\ref{eq_def_potential_well}) holds is clear, because for $w \in H_0^1(\Omega)$, $w \neq 0$, there exists a unique $\lambda > 0$ such that 
\[\lambda \left\Vert w \right\Vert^2_{H_0^1} - \lambda^3 \left\Vert w \right\Vert^4_{L^4} = 0\]
so that if $K(w) = 0$ then $\lambda^\ast(w) = 1$. Note that if $u \in H_0^1(\Omega)$ satisfies $K(u) < 0$, then $\lambda^\ast(u) \in (0, 1)$. Note also that one has
\[\left\vert \int_\Omega \left( w^4 - v^4 \right) \mathrm{d}x \right\vert \lesssim \left\Vert w - v \right\Vert_{L^2} \left( \left\Vert w \right\Vert_{L^6(\Omega)}^3 + \left\Vert v \right\Vert_{L^6(\Omega)}^3 \right)\]
so that $K$ and $J$ are continuous by the Sobolev embedding $H^1(\Omega) \hookrightarrow L^6(\Omega)$. The following result is due to Payne and Sattinger \cite{Payne-Sattinger}, in the case of a bounded subset of $\mathbb{R}^n$. The extension to the case of a compact manifold is straightforward, except for the uniqueness part.

\begin{thm}\label{thm_existence_uniqueness_Q}
One has $d > 0$, and there exists $w \in H_0^1(\Omega)$ such that $w \neq 0$, $K(w) = 0$, and $J(w) = d$. Such a function $w$ is a critical point of $J$, does not change sign, satisfies $w \in H^2(\Omega)$, and $-\Delta w + \beta w = w^3$. In addition, if $\Omega$ is a subset of $\mathbb{R}^3$, then $J$ has exactly two critical points: $w$ and $-w$.
\end{thm}

In what follows, we denote by $Q$ a nonnegative critical point of $J$ and refer to it as a ground state of (\ref{KG_nL}). To our knowledge, the uniqueness of $Q$ has not been proven in the literature in the case of a general compact Riemannian manifold $\Omega$. However, it will not be used anywhere in the present article.

We prove some technical results for later use. The first one states that a function $u$ such that $K(u) < 0$ is far from the zero function.

\begin{lem}\label{lem_lim_inf_J_u_n}
Let $(u_n)$ be a sequence of elements of $H_0^1(\Omega)$ such that $K(u_n) < 0$ for all $n \in \mathbb{N}$, and $K(u_n) \xrightarrow{n \rightarrow \infty} 0$. Then, $\liminf J(u_n) \geq d$.
\end{lem}

\begin{proof}
For $n \in \mathbb{N}$, one has $J(u_n) = \frac{1}{4} K(u_n) + \frac{1}{4} \left\Vert u_n \right\Vert_{H_0^1}^2$, implying
\[\liminf J(u_n) = \frac{1}{4} \liminf \left\Vert u_n \right\Vert_{H_0^1}^2.\]
If $\liminf J(u_n) = + \infty$, the result holds. Otherwise, up to a subsequence, one can assume that $\left( J(u_n) \right)_n$ and $\left( \left\Vert u_n \right\Vert_{H_0^1} \right)_n$ converge. In particular, $(u_n)$ is bounded and we can assume that (up to a subsequence) it converges weakly in $H_0^1(\Omega)$ and strongly in $L^4(\Omega)$. Denote by $u \in H_0^1(\Omega) \cap L^4(\Omega)$ the limit. For $n \in \mathbb{N}$, using the Sobolev embedding and the fact that $K(u_n) \leq 0$, one obtains
\[\left\Vert u_n \right\Vert_{L^4}^4 \lesssim \left\Vert u_n \right\Vert_{H_0^1}^4 \leq \left\Vert u_n \right\Vert_{L^4}^8.\]
As $u_n \neq 0$, this gives $\Vert u \Vert_{L^4} \gtrsim 1$, and $u \neq 0$ as a result. One also has 
\[\Vert u \Vert_{H_0^1}^2 \leq \lim \left\Vert u_n \right\Vert_{H_0^1}^2 \leq \Vert u \Vert_{L^4}^4,\]
yielding $K(u) \leq 0$. We split the end of the proof into two cases.

First, assume that $K(u) = 0$. Then $\Vert u \Vert_{H_0^1} = \lim \left\Vert u_n \right\Vert_{H_0^1}$, and hence, $(u_n)_n$ converges to $u$ strongly in $H_0^1(\Omega)$. In particular, by definition of $d$, one has
\[\lim J(u_n) = J(u) \geq d\]
so the proof is complete in that case.

Second, assume that $K(u) < 0$, that is, 
\[\Vert u \Vert_{H_0^1} < \lim \left\Vert u_n \right\Vert_{H_0^1}.\]
As above, there exists $\lambda^\ast \in (0, 1)$ such that $K\left(\lambda^\ast u \right) = 0$, and $d \leq J \left(\lambda^\ast u \right)$. The fact that $K\left(\lambda^\ast u \right) = 0$ gives
\[J \left(\lambda^\ast u \right) = \frac{1}{4} \left\Vert \lambda^\ast u \right\Vert_{L^4}^4.\]
Writing $J(u_n) = \frac{1}{2} K(u_n) + \frac{1}{4} \left\Vert u_n \right\Vert_{L^4}^4$, one finds $\lim J \left( u_n \right) = \frac{1}{4} \left\Vert u \right\Vert_{L^4}^4$. This gives
\[d \leq J \left(\lambda^\ast u \right) = \left(\lambda^\ast\right)^4 \lim J \left( u_n \right) < \lim J \left( u_n \right),\]
implying the result in the second case.
\end{proof}

The second result will allow us to improve inequalities of the form $K(u) \geq 0$ or $K(u) < 0$ for functions such that $J(u)$ is strictly below $d$.

\begin{lem}\label{lem_improving_K(u)_ineq}
Fix $\delta > 0$. There exists a constant $c > 0$ such that for all $u \in H_0^1(\Omega)$ with $J(u) \leq d - \delta$, one has:
\begin{enumerate}[label=(\roman*)]
\item if $K(u) \geq 0$ then $K(u) \geq c \Vert u \Vert_{H_0^1}^2$.
\item if $K(u) < 0$ then $K(u) \leq - c \Vert u \Vert_{H_0^1}^2$ and $K(u) \leq - c$.
\end{enumerate}
\end{lem}

\begin{rem}
The proof can be carried out by contradiction, but it does not provide any information regarding the size of the constant. Our proof gives the following explicit estimates (which will be of no use in the rest of the article): if $u$ is such that $J(u) \leq d - \delta$ then $K(u) \geq 0$ implies $K(u) \geq \sqrt{\frac{\delta}{d}} \Vert u \Vert_{H_0^1}^2$, and $K(u) < 0$ implies 
\[K(u) \leq - 4 \delta - 4 \sqrt{d \delta} \quad \text{ and } \quad K(u) \leq - \frac{\delta + \sqrt{d \delta}}{d + \sqrt{d \delta}} \Vert u \Vert_{H_0^1}^2.\]
\end{rem}

\begin{proof}
We split the proof into 3 steps.

\paragraph{Step 1: an explicit Sobolev embedding.} We prove that for $u \in H_0^1(\Omega)$, one has 
\begin{equation}\label{eq_proof_lem_improving_K_ineq_1}
\Vert u \Vert_{L^4} \Vert Q \Vert_{L^4} \leq \Vert u \Vert_{H_0^1}.
\end{equation}
Fix $u \in H_0^1(\Omega), u \neq 0$. As above, write $j(\lambda) = J(\lambda u)$ for $\lambda \in \mathbb{R}$. Then 
\begin{equation}\label{eq_proof_lem_improving_K_ineq_2}
\lambda^\ast = \frac{\Vert u \Vert_{H_0^1}}{\Vert u \Vert_{L^4}^2}
\end{equation}
is the positive argument of the maximum of $j$, and one has
\begin{equation}\label{eq_proof_lem_improving_K_ineq_3}
J\left( \lambda^\ast u \right) \geq J(Q) = d.
\end{equation}
Note that $K\left( \lambda^\ast u \right) = K(Q) = 0$ implies 
\[J\left( \lambda^\ast u \right) = \frac{\left\Vert \lambda^\ast u \right\Vert_{L^4}^4}{4} \quad \text{ and } \quad J(Q) = \frac{\Vert Q \Vert_{L^4}^4}{4}.\]
Together with (\ref{eq_proof_lem_improving_K_ineq_2}) and (\ref{eq_proof_lem_improving_K_ineq_3}), this gives (\ref{eq_proof_lem_improving_K_ineq_1}).

\paragraph{Step 2: estimation of $\Vert u \Vert_{H_0^1}$.} Here, we show that there exist $x^+$ and $x^-$ such that $0 < x^- < \Vert Q \Vert_{H_0^1}^2 < x^+$, satisfying the following property: if $u \in H_0^1(\Omega)$ is such that $J(u) \leq d - \delta$, then $K(u) \geq 0$ implies $\Vert u \Vert_{H_0^1} \leq x^-$ and $K(u) < 0$ implies $\Vert u \Vert_{H_0^1} \geq x^+$.

Fix $u \in H_0^1(\Omega)$. For $x \geq 0$, we put 
\[\alpha(x) = \frac{x^2}{2} - \frac{x^4}{4 \Vert Q \Vert_{L^4}^4}\]
so that (\ref{eq_proof_lem_improving_K_ineq_1}) gives $\alpha \left( \Vert u \Vert_{H_0^1} \right) \leq J(u) \leq d - \delta$. The graph of $\alpha$ is given in the following figure. Note that the maximum of $\alpha$ is $d$, and is reached at $x = \Vert Q \Vert_{L^4}^2$.

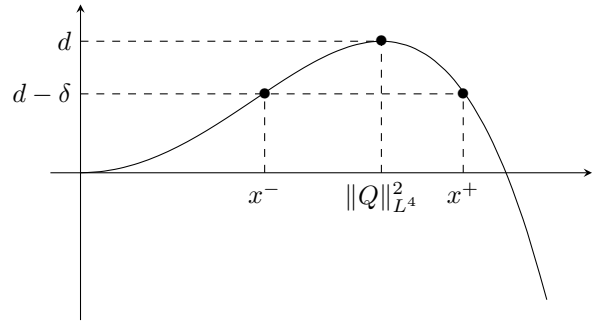
\begin{center}
\begin{tikzpicture}[xscale=4,yscale=7]
\draw[>=stealth,->] (-0.1, 0) -- (1.7, 0);
\draw[>=stealth,->] (0, -0.28) -- (0, 0.32);
\draw[scale=1, domain=0:1.55, smooth, variable=\x] plot ({\x}, {\x*\x/2 - \x*\x*\x*\x/4});

\pgfmathsetmacro{\deltaa}{0.1}
\draw[dashed] (0, 0.25 - \deltaa) -- (1.272, 0.25 - \deltaa);
\draw[dashed] (0, 0.25) -- (1, 0.25);
\draw[dashed] (1, 0) -- (1, 0.25);
\draw[dashed] (0.612, 0) -- (0.612, 0.25 - \deltaa);
\draw[dashed] (1.272, 0) -- (1.272, 0.25 - \deltaa);

\draw (1, 0.25) node {$\bullet$};
\draw (0.612, 0.25 - \deltaa) node {$\bullet$};
\draw (1.272, 0.25 - \deltaa) node {$\bullet$};

\draw (0, 0.25) node[left] {$d$};
\draw (0, 0.25 - \deltaa) node[left] {$d - \delta$};
\draw (1, 0) node[below] {$\Vert Q \Vert_{L^4}^2$};
\draw (1.272, 0) node[below] {$x^+$};
\draw (0.612, 0) node[below] {$x^-$};
\end{tikzpicture}
\end{center}

Note that there exists a unique $x^+ > \Vert Q \Vert_{L^4}^2$ such that $\alpha(x^+) = d - \delta$, and if $\delta \leq d$ then there exists a unique $x^- < \Vert Q \Vert_{L^4}^2$ such that $\alpha(x^-) = d - \delta$. Explicitly, one has
\[x^\pm = 2 \sqrt{d \pm \sqrt{d \delta}}.\]

First, assume that $K(u) \geq 0$. One has 
\[\frac{\Vert u \Vert_{H_0^1}^2}{4} \leq \frac{\Vert u \Vert_{H_0^1}^2}{4} + \frac{K(u)}{4} = J(u) \leq d - \delta,\]
implying $\Vert u \Vert_{H_0^1}^2 \leq 4 d = \Vert Q \Vert_{L^4}^4$. As $\alpha\left( \Vert u \Vert_{H_0^1} \right) \leq d - \delta$, this gives $\Vert u \Vert_{H_0^1} \leq x^-$.

Second, assume that $K(u) < 0$. Then, using (\ref{eq_proof_lem_improving_K_ineq_1}), one obtains
\[\Vert u \Vert_{H_0^1}^2 < \Vert u \Vert_{L^4}^4 \leq \left( \frac{\Vert u \Vert_{H_0^1}}{\Vert Q \Vert_{L^4}} \right)^4,\]
yielding $\Vert u \Vert_{H_0^1} \geq \Vert Q \Vert_{L^4}^2$. As $\alpha\left( \Vert u \Vert_{H_0^1} \right) \leq d - \delta$, this gives $\Vert u \Vert_{H_0^1} \geq x^+$.

\paragraph{Step 3: end of the proof.} First, assume that $K(u) \geq 0$. Then using (\ref{eq_proof_lem_improving_K_ineq_1}) and step 2, we obtain \textit{(i)} as follows
\[K(u) \geq \Vert u \Vert_{H_0^1}^2 - \frac{\Vert u \Vert_{H_0^1}^4}{\Vert Q \Vert_{L^4}^4} \geq \Vert u \Vert_{H_0^1}^2 \left( 1 - \frac{(x^-)^2}{\Vert Q \Vert_{L^4}^4} \right) \geq \sqrt{\frac{\delta}{d}} \Vert u \Vert_{H_0^1}^2.\]

Second, assume that $K(u) < 0$. Using step 2 and $J(u) \leq d - \delta$, one obtains
\[K(u) = 4J(u) - \Vert u \Vert_{H_0^1}^2 \leq 4(d - \delta) - (x^+)^2 = - 4 \delta - 4 \sqrt{d \delta}.\]
That gives the second inequality of \textit{(ii)}. For the first inequality of \textit{(ii)}, it suffices to find a constant $C > 0$ such that 
\[K(u) \leq 4(d - \delta) - \Vert u \Vert_{H_0^1}^2 \leq - C \Vert u \Vert_{H_0^1}^2,\]
using $\Vert u \Vert_{H_0^1} \geq x^+$. It holds if one chooses
\[C \leq 1 - \frac{4(d - \delta)}{(x^+)^2} = \frac{\delta + \sqrt{d \delta}}{d + \sqrt{d \delta}}.\]
The proof is complete.
\end{proof}

\section{The damped equation below the energy of the ground state}

Here, we prove Theorem \ref{thm_dichotomy_damped_introduction}. To begin with, we check that the spaces $\mathscr{K}^+$ and $\mathscr{K}^-$ are invariant under the flow of the damped equation.

\begin{lem}\label{lem_K+-_inv_flow_damped}
Suppose $\left( u^0, u^1 \right) \in \mathscr{K}^\pm$, $\gamma \in L^\infty(\Omega, \mathbb{R}_+)$, and $u$ is the solution of \textnormal{(\ref{KG_nL_damped})}. If $t \geq 0$ is such that $u$ exists on $[0, t]$, then $\left( u(t), \partial_t u(t) \right) \in \mathscr{K}^\pm$.
\end{lem}

\begin{proof}
Assume that $\left( u^0, u^1 \right) \in \mathscr{K}^+$, and that there exists $0 < t_0 \leq t$ such that $\left( u(t_0), \partial_t u(t_0) \right) \notin \mathscr{K}^+$. As $E\left(u(t_0), \partial_t u(t_0) \right) \leq E\left( u^0, u^1 \right) < d$, there exists a sequence $(t_n)_n$ such that $0 < t_n < t$, $K(u(t_n)) < 0$, and $K(u(t_n)) \xrightarrow{n \rightarrow \infty} 0$. Lemma \ref{lem_lim_inf_J_u_n} gives $\liminf J(u(t_n)) \geq d$, a contradiction with
\[J( u(t_n) ) \leq E\left(u(t_n), \partial_t u(t_n) \right) \leq E\left( u^0, u^1 \right) < d, \quad n \in \mathbb{N}.\]

The proof for $\left( u^0, u^1 \right) \in \mathscr{K}^-$ is the same.
\end{proof}

\subsection{Global solutions}

The following result is straightforward, as in the case of the undamped equation.

\begin{thm}
Suppose $\left( u^0, u^1 \right) \in \mathscr{K}^+$, $\gamma \in L^\infty(\Omega, \mathbb{R}_+)$, and $u$ is the solution of \textnormal{(\ref{KG_nL_damped})}. Then $u$ is defined on $\mathbb{R}_+$. 
\end{thm}

\begin{proof} 
Using Lemma \ref{lem_K+-_inv_flow_damped} and the fact that the energy is nonincreasing, one obtains
\begin{align*}
E\left( u^0, u^1 \right) \geq E\left( u(t), \partial_t u(t) \right) & = J(u(t)) + \frac{1}{2} \left\Vert \partial_t u(t) \right\Vert_{L^2}^2 \\
& = \frac{1}{4} K(u(t)) + \frac{1}{4} \left\Vert u(t) \right\Vert_{H_0^1}^2 + \frac{1}{2} \left\Vert \partial_t u(t) \right\Vert_{L^2}^2 \\
& \geq \frac{1}{4} \left\Vert u(t) \right\Vert_{H_0^1}^2 + \frac{1}{2} \left\Vert \partial_t u(t) \right\Vert_{L^2}^2.
\end{align*}
Hence, there is no blow-up, and $u$ is defined on $\mathbb{R}_+$ by the Cauchy theory (see Theorem \ref{thm_existence_damped_waves}). 
\end{proof}

The previous proof also gives the following lemma. 

\begin{lem}[Equivalence between the energy and the square of the $H_0^1 \times L^2$ norm for solutions in $\mathscr{K}^+$]\label{lem_energy_eq_K+}
For all $\left( u^0, u^1 \right) \in \mathscr{K}^+$ and $\gamma \in L^\infty(\Omega, \mathbb{R}_+)$, one has
\[2 E\left( u(t), \partial_t u(t) \right) \leq \left\Vert u(t) \right\Vert_{H_0^1}^2 + \left\Vert \partial_t u(t) \right\Vert_{L^2}^2 \leq 4 E\left( u(t), \partial_t u(t) \right), \quad t \geq 0,\]
where $u$ is the solution of \textnormal{(\ref{KG_nL_damped})} with initial data $\left( u^0, u^1 \right)$.
\end{lem}

We will use this to prove the stabilisation of a solution starting from $\mathscr{K}^+$. one has the following source-to-solution continuity result.

\begin{lem}\label{lem_continuity_source_to_sol}
Fix $T > 0$, $C_0 > 0$, and $\gamma \in L^\infty(\Omega, \mathbb{R}_+)$. There exists a constant $C > 0$ such that for all $\left( u^0, u^1 \right), \left( v^0, v^1 \right) \in H_0^1(\Omega) \times L^2(\Omega)$, if the solutions $u$ and $v$ of \textnormal{(\ref{KG_nL_damped})} with initial data $\left( u^0, u^1 \right)$ and $\left( v^0, v^1 \right)$ are defined on $[0, T]$ and satisfy
\begin{equation}\label{eq_lem_continuity_source_sol_hyp}
\sup_{t \in [0, T]} \left( \left\Vert \left( u(t), \partial_t u(t) \right) \right\Vert_{H_0^1 \times L^2} + \left\Vert \left( v(t), \partial_t v(t) \right) \right\Vert_{H_0^1 \times L^2} \right) \leq C_0,
\end{equation}
then one has
\[\sup_{t \in [0, T]} \left\Vert \left( u(t), \partial_t u(t) \right) - \left( v(t), \partial_t v(t) \right) \right\Vert_{H_0^1 \times L^2} \leq C \left\Vert \left( u(t_0), \partial_t u(t_0) \right) - \left( v(t_0), \partial_t v(t_0) \right) \right\Vert_{H_0^1 \times L^2}\]
for $t_0 \in \left\{0, \frac{T}{2}, T \right\}$.
\end{lem}

\begin{proof}
It suffices to prove the result for $t_0 = 0$ and for $t_0 = T$. In both cases, we can assume that $T$ is arbitrary small: the result for large $T$ follows by iteration. We assume that $t_0 = 0$, the other case is similar.

Set $w = u - v$, solution of 
\[\left \{
\begin{array}{rcccl}
\square w + \gamma \partial_t w + \beta w & = & u^3 - v^3 & \quad & \text{in } \mathbb{R}_+ \times \Omega, \\
(w(0), \partial_t w(0)) & = & \left( u^0, u^1 \right) - \left( v^0, v^1 \right) & \quad & \text{in } \Omega, \\
w & = & 0 & \quad & \text{on } \mathbb{R}_+ \times \partial \Omega.
\end{array}
\right.\]
Assume that $T < 1$. Then, by Theorem \ref{thm_existence_linear_waves} and Remark \ref{rem_constant_linear_estimate}, there exists a constant independent of $T$, $\left( u^0, u^1 \right)$ and $\left( v^0, v^1 \right)$ such that
\[\sup_{t \in [0, T]} \left( \left\Vert w(t) \right\Vert_{H_0^1} + \left\Vert \partial_t w(t) \right\Vert_{L^2} \right) \lesssim \left\Vert \left( u^0, u^1 \right) - \left( v^0, v^1 \right) \right\Vert_{H_0^1 \times L^2} + \left\Vert u^3 - v^3 \right\Vert_{L^1((0, T), L^2)}.\]
Hölder's inequality gives
\[\left\Vert u^3 - v^3 \right\Vert_{L^1((0, T), L^2)} \lesssim T \left\Vert u - v \right\Vert_{L^\infty((0, T), L^6)} \left( \left\Vert u \right\Vert_{L^\infty((0, T), L^6)}^2 + \left\Vert v \right\Vert_{L^\infty((0, T), L^6)}^2 \right).\]
Using (\ref{eq_lem_continuity_source_sol_hyp}) and the Sobolev embedding $H^1(\Omega) \hookrightarrow L^6(\Omega)$, one finds
\begin{align*}
\left\Vert u^3 - v^3 \right\Vert_{L^1((0, T), L^2)} & \lesssim T \left\Vert w \right\Vert_{L^\infty((0, T), H_0^1)} \left( \left\Vert u \right\Vert_{L^\infty((0, T), H_0^1)}^2 + \left\Vert v \right\Vert_{L^\infty((0, T), H_0^1)}^2 \right) \\
& \lesssim T \left\Vert w \right\Vert_{L^\infty((0, T), H_0^1)}. 
\end{align*}
Hence, for $T$ sufficiently small, one obtains
\[\sup_{t \in [0, T]} \left( \left\Vert w(t) \right\Vert_{H_0^1} + \left\Vert \partial_t w(t) \right\Vert_{L^2} \right) \lesssim \left\Vert \left( u^0, u^1 \right) - \left( v^0, v^1 \right) \right\Vert_{H_0^1 \times L^2}\]
and this completes the proof.
\end{proof}

\subsection{Blow-up solutions}

Here, we prove that a solution initiated in $\mathscr{K}^-$ blows up in finite time.

\begin{thm}
For $\gamma \in L^\infty(\Omega, \mathbb{R}_+)$, if $u$ is the solution of \textnormal{(\ref{KG_nL_damped})} with initial data $\left( u^0, u^1 \right) \in \mathscr{K}^-$, then the maximal time of existence of $u$ is finite.
\end{thm}

\begin{proof}
Take $\left( u^0, u^1 \right) \in \mathscr{K}^-$ and assume by contradiction that $u(t)$ exists for all $t \geq 0$. As in the original proof of Payne and Sattinger \cite{Payne-Sattinger}, set $M(t) = \left\Vert u(t) \right\Vert_{L^2}^2$. We prove that $M \in \mathscr{C}^2(\mathbb{R}_+, \mathbb{R})$, with
\begin{equation}\label{eq_proof_blow_up_damped_reg_1}
M^{\prime}(t) = 2 \int_\Omega u(t) \partial_t u(t) \mathrm{d}x.
\end{equation}
and
\begin{equation}\label{eq_proof_blow_up_damped_reg_2}
M^{\prime \prime}(t) = - 2 \left\Vert u(t) \right\Vert_{H_0^1}^2 + 2 \left\Vert u(t) \right\Vert_{L^4}^4 + 2 \left\Vert \partial_t u(t) \right\Vert_{L^2}^2 - 2 \int_\Omega \gamma u(t) \partial_t u(t) \mathrm{d}x
\end{equation}
As $u \in \mathscr{C}^0(\mathbb{R}_+, H_0^1(\Omega)) \cap \mathscr{C}^1(\mathbb{R}_+, L^2(\Omega))$, one finds $M \in \mathscr{C}^1(\mathbb{R}_+, \mathbb{R})$, and (\ref{eq_proof_blow_up_damped_reg_1}) holds true. We show that (\ref{eq_proof_blow_up_damped_reg_2}) holds true in $\mathcal{D}^\prime((0, + \infty), \mathbb{R})$. Take $\phi \in \mathscr{C}_\mathrm{c}^\infty((0, + \infty), \mathbb{R})$, and write
\[\int_0^{+ \infty} M^\prime \phi^\prime \mathrm{d}t = 2 \int_0^{+ \infty} \int_\Omega \left( \partial_t u \partial_t \left( u \phi \right) - \left(\partial_t u \right)^2 \phi \right)\mathrm{d}x \mathrm{d}t.\]
One has $u \phi \in H_0^1((0, + \infty) \times \Omega)$. Consider a sequence $\left( u_n \right)_n$ of elements of $\mathscr{C}_\mathrm{c}^\infty((0, + \infty) \times \Omega, \mathbb{R})$ that converges to $u \phi$ in $H^1((0, + \infty) \times \Omega)$. For $n \in \mathbb{N}$, using the equation satisfied by $u$, one obtains 
\[\int_0^{+ \infty} \int_\Omega \partial_t u \partial_t u_n \mathrm{d}x \mathrm{d}t = \int_0^{+ \infty} \int_\Omega \left( \nabla u \nabla u_n + \gamma \partial_t u u_n + \beta u u_n - u^3 u_n \right)\mathrm{d}x \mathrm{d}t.\]
In the limit $n \rightarrow + \infty$, one obtains
\[\int_0^{+ \infty} \int_\Omega \partial_t u \partial_t \left( u \phi \right) \mathrm{d}x \mathrm{d}t = \int_0^{+ \infty} \int_\Omega \left( \nabla u \nabla \left(u \phi \right) + \gamma u \partial_t u \phi + \beta u^2 \phi - u^4 \phi \right)\mathrm{d}x \mathrm{d}t\]
and as $\phi$ is independent of $x$, this gives
\[\int_0^{+ \infty} M^\prime \phi^\prime \mathrm{d}t = 2 \int_0^{+ \infty} \int_\Omega \left( \left\vert \nabla u \right\vert^2 + \gamma u \partial_t u + \beta u^2 - u^4 - \left(\partial_t u \right)^2 \right) \phi \mathrm{d}x \mathrm{d}t.\]
Hence, the weak derivative of $M^\prime$ exists and is given by (\ref{eq_proof_blow_up_damped_reg_2}). Note that the right-hand side of (\ref{eq_proof_blow_up_damped_reg_2}) depends continuously on $t \geq 0$, so that in fact, $M \in \mathscr{C}^2(\mathbb{R}_+, \mathbb{R})$, and (\ref{eq_proof_blow_up_damped_reg_2}) holds true in a strong sense.

One has 
\[M^{\prime \prime}(t) = - 2 K(u(t)) + 2 \left\Vert \partial_t u(t) \right\Vert_{L^2}^2 - 2 \int_\Omega \gamma u(t) \partial_t u(t) \mathrm{d}x.\]
Write $E(t) = E\left( u(t), \partial_t u(t) \right)$ for the energy. Recall that the energy equality is 
\[E^\prime(t) = - \int_\Omega \gamma \left\vert \partial_t u(t)\right\vert^2 \mathrm{d}x.\]
As the energy is nonincreasing, either it is bounded or it tends to $- \infty$ as $t$ tends to infinity. We treat these two cases separately.

\paragraph{Case 1: bounded energy.} For all $\epsilon > 0$, as $\gamma \in L^\infty(\mathbb{R}_+)$, one has
\[\left\vert 2 \int_\Omega \gamma u(t) \partial_t u(t) \mathrm{d}x \right\vert \lesssim \epsilon \left\Vert u(t) \right\Vert_{H_0^1}^2 + \frac{1}{\epsilon} \int_\Omega \gamma \left\vert \partial_t u(t) \right\vert^2 \mathrm{d}x = \epsilon \left\Vert u(t) \right\Vert_{H_0^1}^2 - \frac{1}{\epsilon} E^\prime(t).\]
This gives 
\begin{equation}\label{eq_proof_thm_blow_up_damped}
M^{\prime \prime}(t) \geq - 2 K(u(t)) - C_1 \epsilon \left\Vert u(t) \right\Vert_{H_0^1}^2 + 2 \left\Vert \partial_t u(t) \right\Vert_{L^2}^2 + \frac{C_1}{\epsilon} E^\prime(t)
\end{equation}
for some $C_1 > 0$. Inequality (\ref{eq_proof_thm_blow_up_damped}) has two consequences.

First, we use Lemma \ref{lem_improving_K(u)_ineq} \emph{(ii)}: for $t \geq 0$, one has
\[J(u(t)) \leq E\left( u(t), \partial_t u(t) \right) \leq E\left( u^0, u^1 \right) < d\]
so there exists $C_2 > 0$ such that
\begin{equation}\label{eq_proof_thm_blow_up_damped_2}
K\left(u (t) \right) \leq - C_2 \left\Vert u(t) \right\Vert_{H_0^1}^2 \quad \text{ and } \quad K\left( u (t) \right) \leq - C_2, \quad t \geq 0.
\end{equation}
Using this in (\ref{eq_proof_thm_blow_up_damped}) and taking $\epsilon$ sufficiently small, one obtains
\[M^{\prime \prime}(t) - \frac{C_1}{\epsilon} E^\prime(t) \gtrsim \left\Vert u(t) \right\Vert_{H_0^1}^2 + \left\Vert \partial_t u(t) \right\Vert_{L^2}^2 \geq \left\Vert u(t) \right\Vert_{H_0^1}^2.\]
Note that by (\ref{eq_proof_thm_blow_up_damped_2}), one has $\left\Vert u(t) \right\Vert_{H_0^1} \gtrsim 1$ for $t \geq 0$. Hence, one has $M^{\prime \prime}(t) - \frac{C_1}{\epsilon} E^\prime(t) \gtrsim 1$, implying
\[M^\prime(t) - \frac{C_1}{\epsilon} E(t) - M^\prime(0) + \frac{C_1}{\epsilon} E(0) \gtrsim t\]
for all $t \geq 0$. As the energy is bounded, this gives 
\[M^\prime(t) \xrightarrow{t \rightarrow + \infty} + \infty \quad \text{ and } \quad M(t) \xrightarrow{t \rightarrow + \infty} + \infty.\]

Second, we use the energy equality. By definition, one has
\[K(u(t)) = 4 E(t) - 2 \left\Vert \partial_t u(t) \right\Vert_{L^2}^2 - \left\Vert u(t) \right\Vert_{H_0^1}^2\]
for all $t \geq 0$. Using this in (\ref{eq_proof_thm_blow_up_damped}), one obtains
\[M^{\prime \prime}(t) - \frac{C_1}{\epsilon} E^\prime(t) \geq 6 \left\Vert \partial_t u(t) \right\Vert_{L^2}^2 + (2 - C_1 \epsilon) \left\Vert u(t) \right\Vert_{H_0^1}^2 - 8 E(t).\]
Set $C = \frac{C_1}{\epsilon}$. As the energy is bounded and as $M(t)$ tends to infinity, for $\epsilon$ chosen sufficiently small and $t$ sufficiently large, one obtains
\[M^{\prime \prime}(t) - C E^\prime(t) \geq 6 \left\Vert \partial_t u(t) \right\Vert_{L^2}^2.\]
Together with the Cauchy-Schwarz inequality, this gives
\[\left\vert M^\prime(t) \right\vert^2 \leq 4 \left( \int_\Omega \left\vert u(t) \right\vert^2 \mathrm{d}x \right) \left( \int_\Omega \left\vert \partial_t u(t) \right\vert^2 \mathrm{d}x \right) \leq \frac{2}{3} M(t) \left( M^{\prime \prime}(t) - C E^\prime(t) \right), \quad t \geq 0 \text{ large}.\]

Dividing by $M(t) M^\prime(t)$, and using $M^\prime(t) \rightarrow + \infty$ and $E^\prime(t) \leq 0$, one writes
\[\frac{M^\prime(t)}{M(t)} \leq \frac{2}{3} \left( \frac{M^{\prime \prime}(t)}{M^\prime(t)} - C \frac{E^\prime(t)}{M^\prime(t)} \right) \leq \frac{2}{3} \left( \frac{M^{\prime \prime}(t)}{M^\prime(t)} - C E^\prime(t) \right)\]
for $t$ large. Fix $T > 0$ such that the previous inequality holds for all $t \geq T$. Integrating, one obtains
\[\ln\left(M(t)\right) - \ln\left( M(T) \right) \leq \frac{2}{3} \left( \ln\left(M^\prime(t)\right) - \ln\left(M^\prime(T)\right) \right) - \frac{2C}{3} \left( E(t) - E(T) \right), \quad t \geq T.\]
As the energy is bounded and $M(t) \rightarrow + \infty$, this gives $\ln\left(M(t)\right) \leq \frac{3}{4} \ln\left(M^\prime(t)\right)$ for $t$ large, yielding
\[\frac{M^\prime(t)}{M(t)^{\frac{4}{3}}} \geq 1, \quad t \geq T^\prime, \quad T^\prime \text{ large}.\]
Integrating between $t_1$ and $t_2$ with $T^\prime \leq t_1 \leq t_2$, one finds
\[- \frac{3}{M(t_2)^{\frac{1}{3}}} + \frac{3}{M(t_1)^{\frac{1}{3}}} \geq t_2 - t_1.\]
Letting $t_2$ tend to infinity gives a contradiction.

\paragraph{Case 2: the energy tends to $- \infty$.} In that case, by definition of the energy, one has 
\[\left\Vert u(t) \right\Vert_{L^4} \xrightarrow{t \rightarrow + \infty} + \infty.\]
For $\epsilon > 0$, as $\gamma \in L^\infty(\mathbb{R}_+)$ and as $\Omega$ is bounded, one has
\[\left\vert 2 \int_\Omega \gamma u(t) \partial_t u(t) \mathrm{d}x \right\vert \lesssim \frac{1}{\epsilon} \left\Vert u(t) \right\Vert_{L^2}^2 + \epsilon \left\Vert \partial_t u(t) \right\Vert_{L^2}^2 \lesssim \frac{1}{\epsilon} \left\Vert u(t) \right\Vert_{L^4}^2 + \epsilon \left\Vert \partial_t u(t) \right\Vert_{L^2}^2,\]
and together with (\ref{eq_proof_blow_up_damped_reg_2}), this gives 
\[M^{\prime \prime}(t) \geq - 2 \left\Vert u(t) \right\Vert_{H_0^1}^2 + (2 - C \epsilon) \left\Vert \partial_t u(t) \right\Vert_{L^2}^2 + 2 \left\Vert u(t) \right\Vert_{L^4}^4 - \frac{C}{\epsilon} \left\Vert u(t) \right\Vert_{L^4}^2\]
for some $C > 0$. Fix $\epsilon > 0$ such that $2 - C \epsilon \geq \frac{3}{2}$. For $t$ sufficiently large, one has
\[2 \left\Vert u(t) \right\Vert_{L^4}^4 - \frac{C}{\epsilon} \left\Vert u(t) \right\Vert_{L^4}^2 \geq \frac{3}{2} \left\Vert u(t) \right\Vert_{L^4}^4,\]
and one obtains
\[M^{\prime \prime}(t) \geq - 2 \left\Vert u(t) \right\Vert_{H_0^1}^2 + \frac{3}{2} \left\Vert \partial_t u(t) \right\Vert_{L^2}^2 + \frac{3}{2} \left\Vert u(t) \right\Vert_{L^4}^4, \quad t \geq 0 \text{ large}.\]
By definition of the energy, one has
\[\left\Vert u(t) \right\Vert^4_{L^4} = 2 \left\Vert \partial_t u(t) \right\Vert^2_{L^2} + 2 \left\Vert u(t) \right\Vert^2_{H_0^1} - 4 E(t)\]
and for $t$ sufficiently large, this gives
\[M^{\prime \prime}(t) \geq \frac{9}{2} \left\Vert \partial_t u(t) \right\Vert_{L^2}^2 + \left\Vert u(t) \right\Vert_{H_0^1}^2 - 6 E(t).\]
In particular, one has $M^{\prime \prime}(t) \rightarrow + \infty$, and consequently $M(t) \rightarrow + \infty$. For $t$ sufficiently large, one also has $M^{\prime \prime}(t) \geq \frac{9}{2} \left\Vert \partial_t u(t) \right\Vert_{L^2}^2$. Using (\ref{eq_proof_blow_up_damped_reg_1}) and the Cauchy-Schwarz inequality, this gives
\[M^\prime(t)^2 \leq \frac{8}{9} M(t) M^{\prime \prime}(t), \quad t \geq 0 \text{ large}.\]
For $\alpha > 0$, set $\tilde{M}(t) = M(t)^{-\alpha}$. For $t$ sufficiently large, one has
\begin{align*}
\tilde{M}^{\prime \prime}(t) 
& = \alpha M(t)^{- \alpha - 2} \left( (\alpha + 1) M^\prime(t)^2 - M(t) M^{\prime \prime}(t) \right) \\ 
& \leq \alpha \left( \frac{8}{9} (\alpha + 1) - 1 \right) M(t)^{- \alpha - 1} M^{\prime \prime}(t)
\end{align*}
Hence, $\tilde{M}$ is a concave function for $t$ sufficiently large if $\alpha$ is chosen sufficiently small. As $\tilde{M} > 0$ and $\tilde{M}(t) \rightarrow 0$, this gives a contradiction.
\end{proof}

\section{Convergence towards a stationary solution}

\subsection{Convergence of a bounded sequence along a time sequence}

Here, we show the following proposition.

\begin{prop}\label{prop_cv_sequence_time}
Fix $\gamma \in L^\infty(\Omega, \mathbb{R}_+)$ satisfying the GCC. Let $\left( u^0_n, u^1_n \right)_{n \in \mathbb{N}}$ be a sequence of elements of $H_0^1(\Omega) \times L^2(\Omega)$, and for $n \in \mathbb{N}$, write $u_n$ for the solution of \textnormal{(\ref{KG_nL_damped})} with initial data $\left( u^0_n, u^1_n \right)$. Let $(T_n)_{n \in \mathbb{N}}$ be a time-sequence satisfying $T_n \rightarrow + \infty$. Assume that each $u_n$ exists on $\mathbb{R}_+$, that there exists $C > 0$ such that 
\begin{equation}\label{eq_prop_cv_sequence_time_1}
\left\Vert u_n(t) \right\Vert_{H_0^1} + \left\Vert \partial_t u_n(t) \right\Vert_{L^2} \leq C, \quad t \geq 0, \quad n \in \mathbb{N},
\end{equation}
and that for all $T > 0$,
\begin{equation}\label{eq_prop_cv_sequence_time_2}
\int_{T_n - T}^{T_n + T} \int_\Omega \gamma \left\vert \partial_t u_n \right\vert^2 \mathrm{d}x \mathrm{d} t \xrightarrow{n \rightarrow \infty} 0.
\end{equation}
Then there exist an increasing function $\phi: \mathbb{N} \rightarrow \mathbb{N}$ and a stationary solution $w$ of \textnormal{(\ref{KG_nL})} such that for all $T > 0$, 
\[\sup_{t \in [-T, T]} \left( \left\Vert u_{\phi(n)}(T_{\phi(n)} + t) - w \right\Vert_{H_0^1} + \left\Vert \partial_t u_{\phi(n)}(T_{\phi(n)} + t) \right\Vert_{L^2} \right) \xrightarrow{n \rightarrow \infty} 0.\]
\end{prop}

\begin{proof}
We divide the proof into three steps.

\paragraph{Step 1: asymptotic compactness.} We use Corollary 4.2 of \cite{Joly-Laurent} (which relies on a result from \cite{DehmanLebeauZuazua}), which we copy here for convenience.

\begin{lem}[Corollary 4.2 of \cite{Joly-Laurent}]
Take $f \in \mathscr{C}^1(\mathbb{R}, \mathbb{R})$ such that
\[f(0) = 0, \quad x f (x) \geq 0, \quad \vert f(x) \vert \lesssim (1 + \vert x \vert )^p, \quad \vert f^\prime(x) \vert \lesssim (1 + \vert x \vert )^{p - 1}\]
with $1 \leq p < 5$. Consider $R > 0$, $T > 0$, $0 \leq s < 1$, and set $\epsilon = \min\left( 1 - s, \frac{5 - p}{2}, \frac{17 - 3 p}{14}\right) > 0$. There exist $C > 0$ and $(q, r)$ satisfying 
\[\frac{1}{q} + \frac{3}{r} = \frac{1}{2}, \quad q \in \left[\frac{7}{2}, + \infty\right]\]
such that the following property holds: if $u \in L^\infty([0, T], H^{1 + s}(\Omega) \cap H_0^1(\Omega))$ satisfies
\[\Vert u \Vert_{L^q([0, T], L^r)} \leq R,\]
then $f(u) \in L^1([0, T], H_0^{s + \epsilon}(\Omega))$, with
\[\left\Vert f(u) \right\Vert_{L^1([0, T], H^{s + \epsilon})} \leq C \left\Vert u \right\Vert_{L^\infty([0, T], H^{1 + s})}.\]
\end{lem}

We will use this lemma with $f(x) = x^3$. Recall that $X^s$ is defined by (\ref{eq_def_X^s}). We prove the following corollary.

\begin{cor}\label{cor_gain_of_regularity_nL}
Consider $T > 0$, $0 \leq s < 1$, $C_0 > 0$, and set $\epsilon = \min\left(1 - s, \frac{4}{7}\right)$. There exists $C > 0$ such that for all $\left( u^0, u^1 \right) \in X^s$, if the solution $u$ of \textnormal{(\ref{KG_nL_damped})} with initial data $\left( u^0, u^1 \right)$ exists on $[0, T]$ and satisfies $u \in L^\infty([0, T], H^{1 + s}(\Omega) \cap H_0^1(\Omega))$, with 
\[\Vert u \Vert_{L^\infty([0, T], H_0^1)} + \left\Vert \partial_t u \right\Vert_{L^\infty([0, T], L^2)} \leq C_0,\]
then one has $u^3 \in L^1([0, T], H^{s + \epsilon}_0(\Omega))$, with
\[\left\Vert u^3 \right\Vert_{L^1([0, T], H_0^{s + \epsilon})} \leq C \left\Vert u \right\Vert_{L^\infty([0, T], H^{1 + s})}.\]
\end{cor}

\begin{proof}
The function $f(x) = x^3$ clearly satisfies the assumption of the previous theorem. By Strichartz estimates (Theorem \ref{thm_strichartz_estimates}), one has
\[\Vert u \Vert_{L^q([0, T], L^r)} \lesssim \Vert u^3 \Vert_{L^1([0, T], L^2)} + \Vert u^0 \Vert_{H_0^1} + \Vert u^1 \Vert_{L^2}.\]
for all $(q, r)$ satisfying $\frac{1}{q} + \frac{3}{r} = \frac{1}{2}$ and $q \in \left[\frac{7}{2}, + \infty\right]$. Using the Sobolev embedding $H^1(\Omega) \hookrightarrow L^6(\Omega)$, one obtains
\[\Vert u \Vert_{L^q([0, T], L^r)} \lesssim \Vert u \Vert_{L^\infty([0, T], H_0^1)}^3 + \Vert u^0 \Vert_{H_0^1} + \Vert u^1 \Vert_{L^2} \leq C_0^3 + C_0.\]
Hence, we can apply Corollary 4.2 of \cite{Joly-Laurent}: it gives
\[\left\Vert u^3 \right\Vert_{L^1([0, T], H_0^{s + \epsilon})} \lesssim \left\Vert u \right\Vert_{L^\infty([0, T], H^{1 + s})}\]
with a constant depending only on $C_0$. This completes the proof of Corollary \ref{cor_gain_of_regularity_nL}.
\end{proof}

We will use this result to find a subsequence of $\left( u_n(T_n), \partial_t u_n(T_n) \right)$ which converges in $H_0^1(\Omega) \times L^2(\Omega)$. Recall that
\[A = 
\begin{pmatrix}
0 & \mathrm{Id} \\
\Delta - \beta & - \gamma
\end{pmatrix}\]
is the infinitesimal generator of the linear part of (\ref{KG_nL_damped}), that $e^{tA}$ is the associated semi-group, and put 
\[U_n = \left(u_n, \partial_t u_n \right) \quad \text{ and } \quad F_n = \left(0, u_n^3 \right)\]
We will use the well-know fact that the GCC implies the stabilisation of the linear version of (\ref{KG_nL_damped}), as stated in Theorem \ref{thm_stab_linear}. Using the Duhamel formula, we can write
\begin{align*}
U_n(T_n) & = e^{T_n A} U_n(0) + \sum_{k = 0}^{\lfloor T_n \rfloor - 1} e^{kA} \int_0^1 e^{s A} F_n(T_n - k - s) \mathrm{d}s + \int_{\lfloor T_n \rfloor}^{T_n} e^{s A} F_n(T_n - s) \mathrm{d}s \\ 
& = e^{T_n A} U_n(0) + \sum_{k = 0}^{\lfloor T_n \rfloor - 1} e^{kA} I_{n, k} + I_n.
\end{align*}

We show that the Duhamel term is bounded in $X^\epsilon$, with $\epsilon = \frac{4}{7}$. To do so, write
\begin{align*}
\Vert I_{n, k} \Vert_{X^\epsilon} & \leq \int_0^1 \left\Vert e^{s A} F_n(T_n - k - s) \right\Vert_{X^\epsilon} \mathrm{d}s \\
& \lesssim \int_0^1 \left\Vert F_n(T_n - k - s) \right\Vert_{X^\epsilon} \mathrm{d}s \\
& = \int_0^1 \left\Vert u_n(T_n - k - s)^3 \right\Vert_{H_0^\epsilon} \mathrm{d}s.
\end{align*}
By (\ref{eq_prop_cv_sequence_time_1}), one has
\[\left\Vert u_n(T_n - k - \cdot ) \right\Vert_{L^\infty((0, 1), H_0^1)} \leq \left\Vert u_n \right\Vert_{L^\infty((0, \infty), H_0^1)} \lesssim 1.\]
Hence, we can apply Corollary \ref{cor_gain_of_regularity_nL} (with $s = 0$) to find $\Vert I_{n, k} \Vert_{X^\epsilon} \lesssim 1$. Similarly, one shows $\Vert I_n \Vert_{X^\epsilon} \lesssim 1$. Using the linear stabilisation (Theorem \ref{thm_stab_linear}), we see that there exists $\lambda > 0$ such that
\[\left\Vert \sum_{k = 0}^{\lfloor T_n \rfloor - 1} e^{kA} I_{n, k} + I_n \right\Vert_{X^\epsilon} \lesssim \sum_{k = 0}^{\lfloor T_n \rfloor - 1} e^{- \lambda k} + 1 \lesssim 1.\]

We have proved that the sequence $\left( U_n(T_n) - e^{T_n A} U_n(0) \right)_n$ is bounded in $X^\epsilon$. By Rellich's theorem, there exists $U_\infty(0) \in H_0^1(\Omega) \times L^2(\Omega)$ such that up to a subsequence, one has
\[U_n(T_n) - e^{T_n A} U_n(0) \xrightarrow{n \rightarrow \infty} U_\infty(0)\]
in $H_0^1(\Omega) \times L^2(\Omega)$. Using linear stabilisation again, one finds $U_n(T_n) \xrightarrow{n \rightarrow \infty} U_\infty(0)$ in $H_0^1(\Omega) \times L^2(\Omega)$. 

Let $u_\infty$ be the solution of (\ref{KG_nL_damped}) with initial data $U_\infty(0)$. We show that $u_\infty$ is defined on $\mathbb{R}$. Fix $T > 0$. For $n$ and $m$ sufficiently large, Lemma \ref{lem_continuity_source_to_sol} gives
\begin{align*}
& \sup_{t \in [-T, T]} \left( \left\Vert u_n(T_n + t) - u_m(T_m + t) \right\Vert_{H_0^1}^2 + \left\Vert \partial_t u_n(T_n + t) - \partial_t u_m(T_m + t) \right\Vert_{L^2}^2 \right) \\
\lesssim & \left\Vert u_n(T_n) - u_m(T_m) \right\Vert_{H_0^1}^2 + \left\Vert \partial_t u_n(T_n) - \partial_t u_m(T_m) \right\Vert_{L^2}^2
\end{align*}
so that $\left( u_n(T_n + \cdot) \right)_n$ is a Cauchy sequence in $\mathscr{C}^0([- T, T], H_0^1(\Omega)) \cap \mathscr{C}^1([- T, T], L^2(\Omega))$. The limit is a solution of (\ref{KG_nL_damped}) on $[-T, T]$, and coincides with $u_\infty$ near $0$. Hence, $u_\infty$ is defined on $\mathbb{R}$. Note that (\ref{eq_prop_cv_sequence_time_1}) gives 
\begin{equation}\label{eq_proof_prop_cv_seq_time_0}
\left\Vert u_\infty(t) \right\Vert_{H_0^1} + \left\Vert \partial_t u_\infty(t) \right\Vert_{L^2} \leq C, \quad t \in \mathbb{R},
\end{equation}
and (\ref{eq_prop_cv_sequence_time_2}) implies 
\[\int_\mathbb{R} \int_\Omega \gamma \left\vert \partial_t u_\infty \right\vert^2 \mathrm{d}x \mathrm{d}t = 0.\]
Hence, one has $E\left(u_\infty(t), \partial_t u_\infty(t) \right) = E\left(u_\infty(0), \partial_t u_\infty(0) \right)$ for all $t \in \mathbb{R}$, and $\partial_t u_\infty(t, x) = 0$ for all $t \in \mathbb{R}$ and almost all $x \in \omega$. 

\paragraph{Step 2: regularity of the limit.}

To ease notations, write $u = u_\infty$. In this step, we will use Corollary \ref{cor_gain_of_regularity_nL} again to show that
\begin{equation}\label{eq_proof_prop_cv_seq_time_1}
\left( u(0), \partial_t u(0) \right) \in H^2(\Omega) \times H_0^1(\Omega)
\end{equation}
and we will use a result of \cite{HaleRaugel} to prove that for all $\alpha \in \left( \frac{1}{2}, 1 \right)$,
\begin{equation}\label{eq_proof_prop_cv_seq_time_2}
u: \mathbb{R} \longrightarrow H^{1 + \alpha}(\Omega) \cap H_0^1(\Omega) \text{ is analytic.} 
\end{equation} 
Finally, we will see that it implies
\begin{equation}\label{eq_proof_prop_cv_seq_time_3}
u \in \mathscr{C}^\infty(\mathbb{R} \times \overline{\Omega}).
\end{equation}

\underline{Proof of (\ref{eq_proof_prop_cv_seq_time_1}).} As above, set $U = \left(u, \partial_t u \right)$ and $F = \left(0, u^3 \right)$. By Duhamel's formula, one has
\begin{align*}
U(t) & = e^{n A} U(t - n) + \int_0^{n} e^{sA} F(t - s) \mathrm{d}s \\ 
& = e^{n A} U(t - n) + \sum_{k = 0}^{n - 1} e^{kA} \int_0^1 e^{s A} F(t - k - s) \mathrm{d}s
\end{align*}
for $t \in \mathbb{R}$ and $n \in \mathbb{N}$. By (\ref{eq_proof_prop_cv_seq_time_0}), linear stabilization gives 
\[e^{n A} U(t - n) \xrightarrow{n \rightarrow \infty} 0, \quad t \in \mathbb{R},\]
in $X^0$. Furthermore, one has
\begin{align*}
\left\Vert \int_0^1 e^{s A} F(t - k - s) \mathrm{d}s \right\Vert_{X^0} & \lesssim \int_0^1 \left\Vert F(t - k - s) \right\Vert_{X^0} \mathrm{d}s \\
& = \left\Vert u^3 \right\Vert_{L^1([t - k - 1, t - k], L^2)}
\end{align*}
implying
\[\left\Vert \int_0^1 e^{s A} F(t - k - s) \mathrm{d}s \right\Vert_{X^0} \lesssim \left\Vert u \right\Vert_{L^\infty([t - k - 1, t - k], H_0^1)}^3 \lesssim 1\]
by the Sobolev embedding $H^1(\Omega) \hookrightarrow L^6(\Omega)$. In particular, using linear stabilization again, one finds
\[\left\Vert e^{kA} \int_0^1 e^{s A} F(t - k - s) \mathrm{d}s \right\Vert_{X^0} \leq C e^{- \lambda k}\]
for some $C > 0$ and $\lambda > 0$. This gives
\begin{equation}\label{eq_proof_prop_cv_seq_time_4}
U(t) = \sum_{k = 0}^\infty e^{kA} \int_0^1 e^{s A} F(t - k - s) \mathrm{d}s, \quad t \in \mathbb{R},
\end{equation}
in $X^0$. 

Next, we prove
\begin{equation}\label{eq_proof_prop_cv_seq_time_5}
U \in L^\infty(\mathbb{R}, X^1)
\end{equation}
by using Corollary \ref{cor_gain_of_regularity_nL} two times. First, by (\ref{eq_proof_prop_cv_seq_time_0}), Corollary \ref{cor_gain_of_regularity_nL} with $T = 1$ and $s = 0$ gives
\[\left\Vert u(t - k - \cdot)^3 \right\Vert_{L^1((0, 1), H_0^{\epsilon})} \lesssim 1, \quad t \in \mathbb{R}, \quad k \in \mathbb{N},\]
where $\epsilon = \frac{4}{7}$. In particular, using linear stabilisation in $X^{\epsilon}$, one has 
\[\left\Vert e^{kA} \int_0^1 e^{s A} F(t - k - s) \mathrm{d}s \right\Vert_{X^{\epsilon}} \lesssim e^{- \lambda k} \left\Vert u(t - k - \cdot)^3 \right\Vert_{L^1((0, 1), H_0^{\epsilon})} \lesssim e^{- \lambda k},\]
implying that equality (\ref{eq_proof_prop_cv_seq_time_4}) holds in $X^{\epsilon}$, and $U \in L^\infty(\mathbb{R}, X^\epsilon)$. Second, Corollary \ref{cor_gain_of_regularity_nL} with $T = 1$ and $s = \frac{4}{7}$ gives
\[\left\Vert u(t - k - \cdot)^3 \right\Vert_{L^1((0, 1), H_0^{s + \epsilon^\prime})} \lesssim 1, \quad t \in \mathbb{R}, \quad k \in \mathbb{N}\]
where $\epsilon^\prime = \frac{3}{7}$. As above, this proves that equality (\ref{eq_proof_prop_cv_seq_time_4}) holds in $X^1$, and that (\ref{eq_proof_prop_cv_seq_time_5}) holds true. In particular, (\ref{eq_proof_prop_cv_seq_time_1}) is true. Note that using the Sobolev embedding $H^2(\Omega) \hookrightarrow \mathscr{C}^0(\overline{\Omega})$ (see for example \cite{Adams-Fournier}, 4.12 Part II with $n = 3$, $m = p = 2$, $j = 0$), one finds $u \in L^\infty(\mathbb{R} \times \overline{\Omega})$.

\underline{Proof of (\ref{eq_proof_prop_cv_seq_time_2}).} Following \cite{Joly-Laurent}, we use Theorem 2.20 of \cite{HaleRaugel} (applied with hypotheses (H3mod) and (H5)), which we copy here for convenience.

\begin{thm}\label{thm_hale_raugel}
Let $Y$ be a Banach space. Let $P_n \in \mathscr{L}(Y)$ be a sequence of continuous linear maps and let $Q_n = \mathrm{Id} - P_n$. Let $A: D(A) \rightarrow Y$ be the generator of a continuous semi-group $e^{t A}$ and let $G \in \mathscr{C}^1(Y, Y)$. Let $U$ be a global solution in $Y$ of
\[\partial_t U(t) = AU(t) + G(U(t)), \quad t \in \mathbb{R}.\]
We further assume that:
\begin{enumerate}[label=(\roman*)]
\item $\{U(t), t \in \mathbb{R}\}$ is contained in a compact set $K$ of $Y$.
\item For any $y \in Y$, $P_n y$ converges to $y$ when $n$ goes to infinity and $(P_n)$ and $(Q_n)$ are sequences of $\mathscr{L}(Y)$ bounded by a constant $C_0$.
\item The operator $A$ splits as $A = A_1 + B_1$, where $B_1$ is bounded and $A_1$ commutes with $P_n$.
\item There exist $M$ and $\lambda > 0$ such that $\left\Vert e^{t A} \right\Vert_{\mathscr{L}(Y)} \leq M e^{- \lambda t}$ for all $t \geq 0$.
\item $G$ is analytic in the ball $B_Y(0,r)$, where $r$ is such that $r \geq 4 C_0 \sup_{t \in \mathbb{R}} \Vert U(t) \Vert_Y$. More precisely,
there exists $\rho > 0$ such that $G$ can be extended to an holomorphic function of $B_Y(0,r) + i B_Y(0, \rho)$.
\item $\{ DG(U(t)) V, t \in \mathbb{R}, \Vert V \Vert_Y \leq 1 \}$ is a relatively compact set of $Y$.
\end{enumerate}
Then the solution $U: \mathbb{R} \rightarrow Y$ is analytic.
\end{thm}

Fix $\alpha \in \left( \frac{1}{2}, 1 \right)$. Write $\left( \lambda_k \right)_{k \in \mathbb{N}}$ for the eigenvalues of the Dirichlet Laplacian, and $\left( \varphi_k \right)_{k \in \mathbb{N}}$ for a basis of normalized eigenvectors. We will apply this theorem with $Y = X^\alpha$. For $n \in \mathbb{N}^\ast$, let $P_n$ be the restriction to $Y$ of the $L^2(M) \times L^2(M)$ orthogonal projection on the vector space generated by the vectors 
\[(\varphi_1, 0), \cdots, (\varphi_n, 0), (0, \varphi_1), \cdots, \text{ and } (0, \varphi_n).\]
By (\ref{eq_X^s_domains_Delta}), we can equip $Y$ with a norm represented by a Fourier series, using the sequence $\left( \varphi_k \right)_{k \in \mathbb{N}}$. This implies $P_n \in \mathscr{L}(Y)$ and $Q_n = \mathrm{Id} - P_n \in \mathscr{L}(Y)$, with 
\begin{equation}\label{eq_proof_prop_cv_seq_time_6}
\left\Vert P_n \right\Vert_{\mathscr{L}(Y)} \leq C_0 \quad \text{and} \quad \left\Vert Q_n \right\Vert_{\mathscr{L}(Y)} \leq C_0
\end{equation}
for some $C_0 > 0$.

Let $G: Y \rightarrow Y$ be given by $G\left( u^0, u^1 \right) = \left( 0, (u^0)^3 \right)$. Note that $G$ is well-defined as $H^{1 + \alpha}(\Omega) \cap H_0^1(\Omega)$ is an algebra, as $\alpha > \frac{1}{2}$. We need to use a smooth version of the damping $\gamma$. Recall that $\omega$ is an open subset of $\Omega$ satisfying the GCC, and that $\gamma(x) > \alpha > 0$ for almost all $x \in \omega$. Recall also that $u$ satisfies $\gamma \partial_t u = 0$ and $\square u + \beta u = u^3$. By continuity of generalized geodesics (see for example Theorem 3.34 of \cite{MelroseSjostrand}), there exists an open subset $\tilde{\omega}$, compactly included in $\omega$, such that the GCC holds for $\tilde{\omega}$. Fix $\tilde{\gamma} \in \mathscr{C}_\mathrm{c}^\infty(\omega, [0, 1])$ such that $\tilde{\gamma} = 1$ on $\tilde{\omega}$. one has $\tilde{\gamma} \partial_t u = 0$, implying
\[\square u + \tilde{\gamma} \partial_t u + \beta u = u^3.\]

Let $\tilde{A}$ be the infinitesimal generator of the linear part of this equation, that is 
\[A = 
\begin{pmatrix}
0 & - 1 \\
- \Delta + \beta & 0
\end{pmatrix}
+
\begin{pmatrix}
0 & 0 \\
0 & \tilde{\gamma}
\end{pmatrix}
= A_1 + B_1.\]

Now, we check that the assumptions of the previous theorem are satisfied. One has $U \in L^\infty(\mathbb{R}, X^1)$, and as $\alpha < 1$ the embedding $X^1 \rightarrow X^\alpha$ is compact by the Rellich theorem, so \emph{(i)} is clear. The fact that for all $y \in Y$, $P_n y$ converges to $y$ when $n$ goes to infinity is obvious when the norm of $Y$ is expressed as a Fourier series. Together with (\ref{eq_proof_prop_cv_seq_time_6}), this gives \emph{(ii)}. As $\tilde{\gamma}$ is smooth, \emph{(iii)} is clear. We have already used the fact that \emph{(iv)} is true (see Theorem \ref{thm_stab_linear}). For $\left( u^0, u^1 \right) \in Y$ and $(v_0, v_1) \in Y$, the function
\[z \in \mathbb{C} \longmapsto G\left( \left( u^0, u^1 \right) + z (v_0, v_1) \right) = \left( 0, (u_0 + z v_0)^3 \right)\]
is well-defined and analytic, as $H^{1 + \alpha}(\Omega) \cap H_0^1(\Omega)$ is an algebra, so \emph{(v)} is true.

Finally, we check that \emph{(vi)} holds. one has
\[\left\{ DG(U(t)) V, t \in \mathbb{R}, \Vert V \Vert_Y \leq 1 \right\} = \left\{\left( 0, u(t)^2 v_0 \right), t \in \mathbb{R}, \Vert v_0 \Vert_{H^{1 + \alpha} \cap H_0^1} \leq 1 \right\}.\]
Take $t \in \mathbb{R}$ and $v_0 \in H^{1 + \alpha}(\Omega) \cap H_0^1(\Omega)$ such that $\Vert v_0 \Vert_{H^{1 + \alpha} \cap H_0^1} \leq 1$. For $\epsilon > 0$ sufficiently small, using again the fact that $H^{1 + \alpha}(\Omega) \cap H_0^1(\Omega)$ is an algebra, one obtainsan write
\[\left\Vert u(t)^2 v_0 \right\Vert_{H_0^{\alpha + \epsilon}} \leq \left\Vert u(t)^2 v_0 \right\Vert_{H^{1 + \alpha} \cap H_0^1} \lesssim \left\Vert u(t) \right\Vert_{H^{1 + \alpha} \cap H_0^1}^2 \left\Vert v_0 \right\Vert_{H^{1 + \alpha} \cap H_0^1} \lesssim \left\Vert u \right\Vert_{L^\infty(\mathbb{R}, X^1)}^2.\]
The embedding $H_0^{\alpha + \epsilon}(\Omega) \rightarrow H_0^\alpha(\Omega)$ is compact by the Rellich theorem, so \emph{(vi)} is true. The previous theorem gives (\ref{eq_proof_prop_cv_seq_time_2}).

\underline{Proof of (\ref{eq_proof_prop_cv_seq_time_3}).} one has $\partial_t^k u(t) \in H^{1 + \alpha}(\Omega) \cap H_0^1(\Omega)$ for all $k \in \mathbb{N}$ and all $\alpha \in \left( \frac{1}{2}, 1 \right)$. This will give (\ref{eq_proof_prop_cv_seq_time_3}) by standard elliptic regularity properties and a bootstrap argument. More precisely, we use the following result, which can be found (for example) in \cite{GilbargTrudinger}, Theorem 9.19. Note that for $\ell \in \mathbb{N}$ and $\lambda \in (0, 1)$, the norm of $\mathscr{C}^{\ell, \lambda}(\overline{\Omega})$ is given by
\[\Vert v \Vert_{\mathscr{C}^{\ell, \lambda}} = \max_{\vert \beta \vert \leq \ell} \left( \Vert \partial_x^\beta v \Vert_{L^\infty(\overline{\Omega})} + \sup_{x, y \in \overline{\Omega}, x \neq y} \frac{\left\vert \partial_x^\beta v(x) - \partial_x^\beta v(y) \right\vert}{\left\vert x - y \right\vert^\lambda} \right), \quad v \in \mathscr{C}^{\ell, \lambda}(\overline{\Omega}).\]

\begin{thm}\label{thm_gilbarg_trudinger}
Let $\mathcal{O}$ be a smooth Riemannian manifold, with or without boundary. Let $u \in H^2(\mathcal{O})$ be such that $\Delta u = f \in \mathscr{C}^{\ell, \lambda}(\overline{\mathcal{O}})$, with $\ell \in \mathbb{N}$ and $\lambda \in (0, 1)$. Then $u \in \mathscr{C}^{\ell + 2, \lambda} \left( \overline{\mathcal{O}} \right)$.
\end{thm}

Set $\alpha = \frac{3}{4}$. We will use the fact that there exists $\lambda \in (0, 1)$ such that the Sobolev embedding $H^{1 + \alpha}(\Omega) \hookrightarrow \mathscr{C}^{0, \lambda}(\overline{\Omega})$ holds true (see for example \cite{Adams-Fournier}, 4.12, Part II with $n = 3$, $p = 2$ and $j = 1$). We prove by induction on $\ell \in \mathbb{N}$ that for all $\phi \in \mathscr{C}_\mathrm{c}^\infty(\mathbb{R})$ and all $k \in \mathbb{N}$, 
\begin{equation}\label{eq_proof_prop_cv_seq_time_7}
\phi \partial_t^k u \in \mathscr{C}^{2 \ell, \lambda} \left( \mathbb{R} \times \overline{\Omega} \right).
\end{equation}

We start with $\ell = 0$. Consider $\phi \in \mathscr{C}_\mathrm{c}^\infty(\mathbb{R})$, $k \in \mathbb{N}$, and let $I \subset \mathbb{R}$ be a compact interval such that $\supp \phi \subset I$. One has $\partial_t^k u$, $\partial_t^{k + 1} u \in \mathscr{C}^0(I, H^{1 + \alpha}(\Omega))$, implying
\[\partial_t^k u, \partial_t^{k + 1} u \in L^\infty(I, \mathscr{C}^{0, \lambda}(\overline{\Omega}))\]
by the Sobolev embedding mentioned above. For $x, y \in \overline{\Omega}$, $t, t^\prime \in I$, one has
\begin{align*}
& \left\vert \phi(t) \partial_t^k u(t, x) - \phi(t^\prime) \partial_t^k u(t^\prime, y) \right\vert \\
\leq \ & \left\vert \phi(t) \right\vert \left\vert \partial_t^k u(t, x) - \partial_t^k u(t, y) \right\vert + \left\vert \int_{t^\prime}^t \partial_t \left( \phi \partial_t^k u \right)(s, y) \mathrm{d}s \right\vert \\
\lesssim \ & \left\Vert \partial_t^k u \right\Vert_{L^\infty(I, \mathscr{C}^{0, \lambda})} \left\vert x - y \right\vert^\lambda + \left( \left\Vert \partial_t^k u \right\Vert_{L^\infty(I, \mathscr{C}^{0, \lambda})} + \left\Vert \partial_t^{k + 1} u \right\Vert_{L^\infty(I, \mathscr{C}^{0, \lambda})} \right) \left\vert t - t^\prime \right\vert
\end{align*}
yielding (\ref{eq_proof_prop_cv_seq_time_7}) for $\ell = 0$. 

Now, let $\ell \in \mathbb{N}$ be such that (\ref{eq_proof_prop_cv_seq_time_7}) holds true. As above, consider $\phi \in \mathscr{C}_\mathrm{c}^\infty(\mathbb{R})$, $k \in \mathbb{N}$, and $I \subset \mathbb{R}$ a compact interval such that $\supp \phi \subset I$. One has
\[\left( \partial_t^2 + \Delta\right) \left( \phi \partial_t^k u \right) = \partial_t^2 \left( \phi \partial_t^k u \right) + \phi \partial_t^k \left( \partial_t^2 u + \beta u - u^3 \right).\]
As $\mathscr{C}^{2 \ell, \lambda}(\mathbb{R} \times \overline{\Omega})$ is an algebra, (\ref{eq_proof_prop_cv_seq_time_7}) gives $\left( \partial_t^2 + \Delta\right) \left( \phi \partial_t^k u \right) \in \mathscr{C}^{2 \ell, \lambda}(\mathbb{R} \times \overline{\Omega})$. One also has 
\[\left( \partial_t^2 + \Delta\right) \left( \phi \partial_t^k u \right) \in L^2(I \times \Omega),\]
yielding $\phi \partial_t^k u \in H^2(I \times \Omega)$. By Theorem \ref{thm_gilbarg_trudinger}, applied in a smooth open subset of $I \times \Omega$ containing $\supp \phi \times \Omega$, one obtains (\ref{eq_proof_prop_cv_seq_time_7}) for $\ell + 1$. This proves (\ref{eq_proof_prop_cv_seq_time_3}).

\paragraph{Step 3: identification of the limit.} Here, we complete the proof of Proposition \ref{prop_cv_sequence_time} by showing that $u$ is a stationary solution of (\ref{KG_nL}). We use Corollary 3.2 of \cite{Joly-Laurent} (which is a consequence of Theorem A of \cite{RobbianoZuily}), which we copy here for convenience.

\begin{thm}[Corollary 3.2 of \cite{Joly-Laurent}]
Let $T \in (0, + \infty]$ and let $b$, $(c_i)_{i=1,2,3}$ and $d$ be coefficients in $\mathscr{C}^\infty(\Omega \times [0, T), \mathbb{R})$. Assume moreover that $b$, $c$ and $d$ are analytic in time and that $u$ is a strong solution of
\[\partial_t^2u = \Delta u + b \partial_t u + c \cdot \nabla u + d u, \quad (t, x) \in (-T, T) \times \Omega.\]
Let $\mathcal{O}$ be a nonempty open subset of $\Omega$ such that $u(x, t) = 0$ in $\mathcal{O} \times (-T, T)$. Then $u(x, 0) = 0$ in 
\[\mathcal{O}_T = \left\{x \in \Omega, \mathrm{d}(x, \mathcal{O}) < T \right\}\]
where the distance $\mathrm{d}(x, \mathcal{O})$ is defined as the infimum of the lengths of the $\mathscr{C}^1-$paths between $x$ and a point of $\mathcal{O}$.
\end{thm}

We apply this theorem with $\mathcal{O} = \omega$ and $T = +\infty$: the GCC implies that $\mathcal{O}_T = \Omega$ is that case. As $u \in \mathscr{C}^\infty(\mathbb{R} \times \Omega)$, if $v = \partial_t u$ then one has
\[\partial_t^2 v = \Delta v - \beta v + 3 u^2 v.\]
As $u$ is smooth and analytic in time, and as $v = 0$ on $\omega$, one obtains $v = 0$. Hence, $u$ is a stationary solution of (\ref{KG_nL}), and this completes the proof of Proposition \ref{prop_cv_sequence_time}.
\end{proof}

\subsection{Convergence of a global solution}

Here, we prove Theorem \ref{thm_cv_global_solution}. Let $\left( u^0, u^1 \right) \in H_0^1(\Omega) \times L^2(\Omega)$ be such that the solution $u$ of \textnormal{(\ref{KG_nL_damped})} with initial data $\left( u^0, u^1 \right)$ exists on $\mathbb{R}_+$. In a companion paper \cite{perrinUniform}, we prove that there exists $C > 0$ such that 
\begin{equation}\label{eq_proof_thm_cv_bounded_sol_1}
\left\Vert u(t) \right\Vert_{H_0^1} + \left\Vert \partial_t u(t) \right\Vert_{L^2} \leq C, \quad t \geq 0.
\end{equation}
By (\ref{eq_proof_thm_cv_bounded_sol_1}), the energy of $u$ is bounded. Hence, as the energy is nonincreasing, there exists $E_\infty \leq E\left( u^0, u^1 \right)$ such that
\[E\left( u(t), \partial_t u(t) \right) \xrightarrow{t \rightarrow + \infty} E_\infty.\]
The energy equality gives
\[E\left( u^0, u^1 \right) - E_\infty = \int_0^\infty \int_\Omega \gamma \left\vert \partial_t u \right\vert^2 \mathrm{d}x \mathrm{d} t < \infty\]
and this allows us to use Proposition \ref{prop_cv_sequence_time}. More precisely, let $(T_n)_{n \in \mathbb{N}}$ be a time-sequence satisfying $T_n \rightarrow + \infty$. For all $T > 0$, one has
\[\int_{T_n - T}^{T_n + T} \int_\Omega \gamma \left\vert \partial_t u \right\vert^2 \mathrm{d}x \mathrm{d} t \xrightarrow{n \rightarrow \infty} 0.\]
Using also (\ref{eq_proof_thm_cv_bounded_sol_1}), we see that all the assumptions of Proposition \ref{prop_cv_sequence_time} are satisfied by the constant sequence $\left( u^0_n, u^1_n \right)_{n \in \mathbb{N}} = \left( u^0, u^1 \right)_{n \in \mathbb{N}}$. Hence, there exist an increasing function $\phi: \mathbb{N} \rightarrow \mathbb{N}$ and a stationary solution $w$ of \textnormal{(\ref{KG_nL})}, both depending on the sequence $(T_n)$, such that for all $T > 0$, 
\[\sup_{t \in [-T, T]} \left( \left\Vert u(T_{\phi(n)} + t) - w \right\Vert_{H_0^1} + \left\Vert \partial_t u(T_{\phi(n)} + t) \right\Vert_{L^2} \right) \xrightarrow{n \rightarrow \infty} 0.\]
In particular, this implies
\[J(w) = \lim_{n \rightarrow \infty} E\left( u(T_{\phi(n)}), \partial_t u(T_{\phi(n)}) \right) = E_\infty.\]

Now, we assume that 
\begin{equation}\label{eq_proof_thm_cv_bounded_sol_2}
\text{there is an at most countable number of stationary solution } w \text{ such that } J(w) = E_\infty.
\end{equation}
We show that this implies that $w$ is, in fact, independent of $(T_n)$. Assume by contradiction that there exist two sequences $(T_n)$ and $(T_n^\prime)$ such that for all $T > 0$, one has
\[\sup_{t \in [-T, T]} \left( \left\Vert u(T_n + t) - w \right\Vert_{H_0^1} + \left\Vert \partial_t u(T_n + t) \right\Vert_{L^2} \right) \xrightarrow{n \rightarrow \infty} 0\]
and
\[\sup_{t \in [-T, T]} \left( \left\Vert u(T_n^\prime + t) - w^\prime \right\Vert_{H_0^1} + \left\Vert \partial_t u(T_n^\prime + t) \right\Vert_{L^2} \right) \xrightarrow{n \rightarrow \infty} 0\]
where $w$ and $w^\prime$ are two distinct stationary solutions of (\ref{KG_nL}). Fix $\varphi \in L^2(\Omega)$ such that
\[\int_\Omega w \varphi \mathrm{d}x < \int_\Omega w^\prime \varphi \mathrm{d}x.\]
For $t \geq 0$, set $\alpha(t) = \int_\Omega u(t, x) \varphi(x) \mathrm{d}x$. The function $\alpha$ is real and continuous. One has
\[\lim_{n \rightarrow \infty} \alpha(T_n) = \int_\Omega w \varphi \mathrm{d}x < \int_\Omega w^\prime \varphi \mathrm{d}x = \lim_{n \rightarrow \infty} \alpha(T_n^\prime)\]
so by the intermediate value theorem, for all $\ell \in \left[ \int_\Omega w \varphi \mathrm{d}x, \int_\Omega w^\prime \varphi \mathrm{d}x \right]$, there exists a sequence $(t_n)_n$ such that $t_n \rightarrow + \infty$ and
\[\alpha(t_n) \xrightarrow{n \rightarrow \infty} \ell.\]
Hence, for all such $\ell$, there exists a stationary solution $w^{\prime \prime}$ such that $J(w^{\prime \prime}) = E_\infty$ and
\[\ell = \int_\Omega w^{\prime \prime} \varphi \mathrm{d}x.\]
This is a contradiction with (\ref{eq_proof_thm_cv_bounded_sol_2}).

Summarizing, we have proved that there exists a stationary solution $w$ such that for all sequence $(T_n)_n$ such that $T_n \rightarrow + \infty$, there exists a subsequence $(T_{\phi(n)})_n$ such that for all $T > 0$, one has
\[\sup_{t \in [-T, T]} \left( \left\Vert u(T_{\phi(n)} + t) - w \right\Vert_{H_0^1} + \left\Vert \partial_t u(T_{\phi(n)} + t) \right\Vert_{L^2} \right) \xrightarrow{n \rightarrow \infty} 0.\]
A basic contradiction argument ends the proof of Theorem \ref{thm_cv_global_solution}.

\section{Stabilisation of global solutions below the ground state energy}

\subsection{Case of a positive damping}

In the case of a positive damping, it is possible to make a short proof of the stabilization of solutions of (\ref{KG_nL_damped}) initiated in $\mathscr{K}^+$. That proof is based on the arguments of the proof of Proposition 2.5 of \cite{Joly-Laurent} and on Lemma \ref{lem_improving_K(u)_ineq}.

\begin{thm}\label{thm_positive damping}
Assume that $\gamma(x) \geq \alpha > 0$ on $\Omega$. Then for any $E_0 \in [0, d)$, there exists $C > 0$ and $\lambda > 0$ such that for all $\left( u^0, u^1 \right) \in \mathscr{K}^+$ such that $E\left( u^0, u^1 \right) \in [0, E_0]$, if $u$ is the solution of \textnormal{(\ref{KG_nL_damped})} with initial data $\left( u^0, u^1 \right)$, then
\[\left\Vert u(t) \right\Vert_{H_0^1} + \left\Vert \partial_t u(t) \right\Vert_{L^2} \leq C e^{-\lambda t}, \quad t \geq 0.\]
\end{thm}

\begin{rem}\label{rem_constant_dependance}
Note that $C$ and $\lambda$ depend on $E_0$. The following result is false: there exist $C > 0$ and $\lambda > 0$ such that for $\left( u^0, u^1 \right) \in \mathscr{K}^+$ and $t \geq 0$, one has
\begin{equation}\label{eq_rem_constant_dependance}
\left\Vert u(t) \right\Vert_{H_0^1} + \left\Vert \partial_t u(t) \right\Vert_{L^2} \leq C e^{-\lambda t}.
\end{equation}
Indeed, for $n \in \mathbb{N}^\ast$, consider $\left( u_n^0, u_n^1 \right) = \left( \left(1 - \frac{1}{n} \right) Q, 0 \right)$ and write $u_n$ for the solution with initial data $\left( u_n^0, u_n^1 \right)$. Recall that if $j(\lambda) = J(\lambda Q)$, then one has $j^\prime(\lambda) > 0$ for $\lambda \in (0, 1)$ and $j^\prime(\lambda) < 0$ for $\lambda > 1$. In particular, this implies $K(u_n^0) > 0$ and $\left( u_n^0, u_n^1 \right) \in \mathscr{K}^+$. Hence, if (\ref{eq_rem_constant_dependance}) is true, then there exists $T > 0$ such that for all $n \in \mathbb{N}^\ast$, one has
\[\left\Vert u_n(T) \right\Vert_{H_0^1} \leq \frac{1}{2} \left\Vert Q \right\Vert_{H_0^1}.\]
By continuity of the source-to-solution operator (see Lemma \ref{lem_continuity_source_to_sol}), one has 
\[\left\Vert u_n(T) - Q \right\Vert_{H_0^1} \xrightarrow{n \rightarrow \infty} 0\]
and this is a contradiction.
\end{rem}

\begin{proof}
Fix $E_0 \in [0, d)$, and $\left( u^0, u^1 \right) \in \mathscr{K}^+$ such that $E\left(u^0, u^1 \right) \leq E_0$. For $\epsilon > 0$, we define
\[E_\epsilon(t) = E\left(u(t), \partial_t u(t) \right) + \epsilon \int_\Omega u(t, x) \partial_t u(t, x) \mathrm{d}x.\]
Recall that as $K(u(t)) \geq 0$, one has 
\[E\left(u(t), \partial_t u(t) \right) \geq \frac{1}{4} \left\Vert u(t) \right\Vert_{H_0^1}^2 + \frac{1}{2} \left\Vert \partial_t u(t) \right\Vert_{L^2}^2, \quad t \geq 0.\]
For $\epsilon > 0$ sufficiently small, one finds
\[\epsilon \left\vert \int_\Omega u(t, x) \partial_t u(t, x) \mathrm{d}x \right\vert \leq \frac{1}{2} E\left(u(t), \partial_t u(t) \right)\]
implying
\begin{equation}\label{eq_proof_thm_stab_positive_damping}
\frac{1}{2} E\left(u(t), \partial_t u(t) \right) \leq E_\epsilon(t) \leq \frac{3}{2} E\left(u(t), \partial_t u(t) \right), \quad t \geq 0.
\end{equation}
Using the energy equality, for $\epsilon > 0$ and $t \geq 0$, one obtains
\[E_\epsilon^\prime(t) = \int_\Omega (\epsilon - \gamma) \left\vert \partial_t u(t) \right\vert^2 \mathrm{d} x - \epsilon \left\Vert u(t) \right\Vert_{H_0^1}^2 + \epsilon \left\Vert u(t) \right\Vert_{L^4}^4 - \epsilon \int_\Omega \gamma u(t) \partial_t u(t) \mathrm{d} x.\]
As $\gamma \geq \alpha$, this gives
\[E_\epsilon^\prime(t) \leq (\epsilon - \alpha) \left\Vert \partial_t u(t) \right\Vert_{L^2}^2 - \epsilon \left\Vert u(t) \right\Vert_{H_0^1}^2 + \epsilon \left\Vert u(t) \right\Vert_{L^4}^4 - \epsilon \int_\Omega \gamma u(t) \partial_t u(t) \mathrm{d} x.\]
Note that for $t \geq 0$, one has
\[J(u(t)) \leq E\left( u(t), \partial_t u(t) \right) \leq E\left( u^0, u^1 \right) \leq E_0 < d\]
so by Lemma \ref{lem_improving_K(u)_ineq} \emph{(i)}, there exists $C > 0$ such that $K(u(t)) \geq C \left\Vert u(t) \right\Vert_{H_0^1}^2$ for all $t \geq 0$. This gives 
\[E_\epsilon^\prime(t) \leq (\epsilon - \alpha) \left\Vert \partial_t u(t) \right\Vert_{L^2}^2 - \epsilon C \left\Vert u(t) \right\Vert_{H_0^1}^2 - \epsilon \int_\Omega \gamma u(t) \partial_t u(t) \mathrm{d} x.\]
Writing 
\[\epsilon \left\vert \int_\Omega u(t, x) \partial_t u(t, x) \mathrm{d}x \right\vert \lesssim \epsilon^{\frac{3}{2}} \left\Vert u(t) \right\Vert_{H_0^1}^2 + \epsilon^{\frac{1}{2}} \left\Vert \partial_t u(t) \right\Vert_{L^2}^2\]
one finds
\[E_\epsilon^\prime(t) \leq (\epsilon - \alpha + C^\prime \sqrt{\epsilon}) \left\Vert \partial_t u(t) \right\Vert_{L^2}^2 + \epsilon \left(C^\prime \sqrt{\epsilon} - C \right) \left\Vert u(t) \right\Vert_{H_0^1}^2\]
for some $C^\prime > 0$. Fix $T > 0$. For $\epsilon > 0$ sufficiently small, one has 
\[E_\epsilon(T) - E_\epsilon(0) \lesssim - \int_0^T \left( \left\Vert \partial_t u(t) \right\Vert_{L^2}^2 + \left\Vert u(t) \right\Vert_{H_0^1}^2\right) \mathrm{d}t,\]
and using (\ref{eq_proof_thm_stab_positive_damping}), one finds
\[E\left(u(T), \partial_t u(T) \right) - 3 E\left(u(0), \partial_t u(0) \right) \lesssim - \int_0^T E\left(u(t), \partial_t u(t) \right) \mathrm{d}t.\]
As the energy is decreasing, this gives 
\[E\left(u(T), \partial_t u(T) \right) - 3 E\left(u(0), \partial_t u(0) \right) \lesssim - T E\left(u(T), \partial_t u(T) \right).\]
Choosing $T$ sufficiently large, one obtains
\[E\left(u(T), \partial_t u(T) \right) \leq \mu E\left(u(0), \partial_t u(0) \right)\]
for some $\mu \in (0, 1)$. As $\left( u(t), \partial_t u(t) \right) \in \mathscr{K}^+$ for all $t \geq 0$, we can iterate this process to get
\[E\left(u(nT), \partial_t u(nT) \right) \leq \mu^n E\left(u^0, u^1 \right), \quad n \in \mathbb{N}.\]
This completes the proof.
\end{proof}

\subsection{Case of a damping satisfying the GCC}

Here, we prove Theorem \ref{thm_stab_GCC_introduction}, using Proposition \ref{prop_cv_sequence_time}. We spit the proof into two steps.

\paragraph{Step 1: observability estimate and application of Proposition \ref{prop_cv_sequence_time}.}
To prove Theorem \ref{thm_stab_GCC_introduction}, it suffices to show the following observability inequality: there exist $C > 0$ and $T > 0$ such that for all $\left( u^0, u^1 \right) \in \mathscr{K}^+$ with $E\left(u^0, u^1 \right) \leq E_0$, the solution $u$ of (\ref{KG_nL_damped}) with initial data $\left(u^0, u^1 \right)$ satisfies
\[E\left(u^0, u^1 \right) \leq C \int_0^T \int_\Omega \gamma(x) \left\vert \partial_t u(t, x) \right\vert^2 \mathrm{d}x \mathrm{d}t.\]
Indeed, if that observability inequality holds, then using the energy equality, one obtains
\[E\left(u(T), \partial_t u(T) \right) \leq \left( 1 - \frac{1}{C} \right) E\left(u^0, u^1 \right).\]
As $\left( u(t), \partial_t u(t) \right) \in \mathscr{K}^+$ for all $t \geq 0$, we can iterate this process to get
\[E\left(u(nT), \partial_t u(nT) \right) \leq \left( 1 - \frac{1}{C} \right)^n E\left(u^0, u^1 \right), \quad n \in \mathbb{N}.\]
This proves that $t \mapsto E\left(u(t), \partial_t u(t) \right)$ decays exponentially. As $K(u(t)) \geq 0$, one has 
\[E\left(u(t), \partial_t u(t) \right) \geq \frac{1}{4} \left\Vert u(t) \right\Vert_{H_0^1}^2 + \frac{1}{2} \left\Vert \partial_t u(t) \right\Vert_{L^2}^2\]
and this gives the conclusion. 

We prove the observability inequality by contradiction. We assume that there exists a sequence $(T_n)_{n \geq 1}$, satisfying $T_n \rightarrow + \infty$, such that for all $n \in \mathbb{N}$, $n \geq 1$, there exists $\left( u^0_n, u^1_n \right) \in \mathscr{K}^+$ such that $E\left( u^0_n, u^1_n \right) \leq E_0$ and
\begin{equation}\label{eq_proof_thm_stab_GCC_1}
\int_0^{T_n} \int_\Omega \gamma(x) \left\vert \partial_t u_n(t, x) \right\vert^2 \mathrm{d}x \mathrm{d}t < \frac{1}{n} E\left(u_n^0, u_n^1 \right).
\end{equation}
For $n \geq 1$ and $t \geq 0$, as $K(u_n(t)) \geq 0$, one has 
\[\left\Vert u_n(t) \right\Vert_{H_0^1}^2 + \left\Vert \partial_t u_n(t) \right\Vert_{L^2}^2 \lesssim E\left(u_n(t), \partial_t u_n(t) \right) \leq E\left(u_n^0, u_n^1 \right) \leq E_0.\]
Write $T_n^\prime = \frac{T_n}{2}$. For all $T > 0$, if $n \in \mathbb{N}$ is sufficiently large, then 
\[\int_{T_n^\prime - T}^{T_n^\prime + T} \int_\Omega \gamma \left\vert \partial_t u_n \right\vert^2 \mathrm{d}x \mathrm{d} t \leq \int_0^{T_n} \int_\Omega \gamma \left\vert \partial_t u_n \right\vert^2 \mathrm{d}x \mathrm{d}t.\]
This gives 
\[\int_{T_n^\prime - T}^{T_n^\prime + T} \int_\Omega \gamma \left\vert \partial_t u_n \right\vert^2 \mathrm{d}x \mathrm{d} t \xrightarrow{n \rightarrow \infty} 0, \quad T > 0.\]
Hence, we can apply Proposition \ref{prop_cv_sequence_time}: there exist an increasing function $\phi: \mathbb{N} \rightarrow \mathbb{N}$ and a stationary solution $w$ of \textnormal{(\ref{KG_nL})} such that
\[\sup_{t \in [-T, T]} \left( \left\Vert u_{\phi(n)}(T_{\phi(n)}^\prime + t) - w \right\Vert_{H_0^1} + \left\Vert \partial_t u_{\phi(n)}(T_{\phi(n)}^\prime + t) \right\Vert_{L^2} \right) \xrightarrow{n \rightarrow \infty} 0, \quad T > 0.\]

As the energy is nondecreasing, one has
\[J(w) = \lim_{n \rightarrow \infty} E\left(u_{\phi(n)}(T_{\phi(n)}^\prime), \partial_t u_{\phi(n)}(T_{\phi(n)}^\prime) \right) \leq E_0 < d.\]
Hence, $w$ is a stationary solution satisfying $J(w) \in [0, d)$: this gives $w = 0$. Using the energy equality, one has
\[E\left(u_n(T_n^\prime), \partial_t u_n(T_n^\prime) \right) \geq E\left(u_n(T_n), \partial_t u_n(T_n) \right) = E\left(u_n^0, u_n^1 \right) - \int_0^{T_n} \int_\Omega \gamma(x) \left\vert \partial_t u_n(t, x) \right\vert^2 \mathrm{d}x \mathrm{d}t\]
so that (\ref{eq_proof_thm_stab_GCC_1}) gives
\[E\left(u_n(T_n^\prime), \partial_t u_n(T_n^\prime) \right) \geq \left( 1 - \frac{1}{n} \right) E\left(u_n^0, u_n^1 \right).\]
In particular, one obtains
\[E\left(u_{\phi(n)}^0, u_{\phi(n)}^1 \right) \xrightarrow{n \rightarrow \infty} 0.\]
For $n \in \mathbb{N}^\ast$, set $\alpha_n = \sqrt{E\left(u_n^0, u_n^1 \right)}$. In the next step, we will consider the equation scaled by $\alpha_n$ to get a contradiction.

\paragraph{Step 2: the scaled equation.} For $n \in \mathbb{N}$, set $w_n = \frac{u_n}{\alpha_n}$. For $n \in \mathbb{N}^\ast$, $w_n$ is the solution of
\[\square w_n + \beta w_n = \alpha_n^2 w_n^3\]
and one has
\[\left\Vert w_n(T_n) \right\Vert_{H_0^1}^2 + \left\Vert \partial_t w_n(T_n) \right\Vert_{L^2}^2 = \frac{1}{\alpha_n^2} \left(\left\Vert u_n(T_n) \right\Vert_{H_0^1}^2 + \left\Vert \partial_t u_n(T_n) \right\Vert_{L^2}^2 \right) \gtrsim \frac{E\left(u_n(T_n), \partial_t u_n(T_n) \right)}{\alpha_n^2}.\]
Using the energy equality and (\ref{eq_proof_thm_stab_GCC_1}), one finds
\begin{equation}\label{eq_proof_thm_stab_GCC_2}
\left\Vert w_n(T_n) \right\Vert_{H_0^1}^2 + \left\Vert \partial_t w_n(T_n) \right\Vert_{L^2}^2 \gtrsim \frac{1}{\alpha_n^2} \left(E\left(u_n^0, u_n^1 \right) - \frac{1}{n} E\left(u_n^0, u_n^1 \right) \right) = 1 - \frac{1}{n} \geq \frac{1}{2}, \quad n \geq 2.
\end{equation}

Recall that $A$ denotes the infinitesimal generator of the linear part of (\ref{KG_nL_damped}), that $e^{tA}$ is the associated semi-group, and set 
\[W_n = \left(w_n, \partial_t w_n \right) \quad \text{ and } \quad F_n = \left(0, \alpha_n^2 w_n^3 \right).\]
Using the Duhamel formula, we can write
\begin{align}
W_n(T_n) & = e^{T_n A} W_n(0) + \int_0^{T_n} e^{(T_n - s) A} F_n(s) \mathrm{d}s \nonumber \\ 
& = e^{T_n A} W_n(0) + \sum_{k = 0}^{\lfloor T_n \rfloor - 1} e^{(T_n - k) A} \int_0^1 e^{- s A} F_n(k + s) \mathrm{d}s + \int_{\lfloor T_n \rfloor}^{T_n} e^{(T_n - s) A} F_n(s) \mathrm{d}s. \label{eq_proof_thm_stab_GCC_3} 
\end{align}
As $K(u_n(0)) \geq 0$, one has 
\[\left\Vert w_n(0) \right\Vert_{H_0^1}^2 + \left\Vert \partial_t w_n(0) \right\Vert_{L^2}^2 = \frac{1}{\alpha_n^2} \left(\left\Vert u_n(0) \right\Vert_{H_0^1}^2 + \left\Vert \partial_t u_n(0) \right\Vert_{L^2}^2 \right) \lesssim \frac{E\left(u_n(0), \partial_t u_n(0) \right)}{\alpha_n^2} = 1.\]
Hence, the sequence $\left( W_n(0) \right)$ is bounded, and linear stabilisation (Theorem \ref{thm_stab_linear}) gives
\[e^{T_n A} W_n(0) \xrightarrow{n \rightarrow \infty} 0\]
in $H_0^1(\Omega) \times L^2(\Omega)$. Next, write
\[\left\Vert \int_0^1 e^{- s A} F_n(k + s) \mathrm{d}s \right\Vert_{H_0^1 \times L^2} \lesssim \left\Vert F_n(k + \cdot) \right\Vert_{L^1((0, 1), H_0^1 \times L^2)} = \left\Vert \alpha_n^2 w_n^3 \right\Vert_{L^1((k, k + 1), L^2)}.\]
A Sobolev embedding gives
\[\left\Vert \int_0^1 e^{- s A} F_n(k + s) \mathrm{d}s \right\Vert_{H_0^1 \times L^2} \lesssim \frac{1}{\alpha_n} \left\Vert u_n \right\Vert_{L^3((k, k + 1), L^6)}^3 \lesssim \frac{1}{\alpha_n} \left\Vert \left(u_n, \partial_t u_n \right) \right\Vert_{L^\infty((k, k+1), H_0^1 \times L^2)}^3.\]
Using $K(u_n) \geq 0$ and the fact that the energy is nonincreasing, one obtains
\[\left\Vert \int_0^1 e^{- s A} F_n(k + s) \mathrm{d}s \right\Vert_{H_0^1 \times L^2} \lesssim \frac{E\left(u_n^0, u_n^1 \right)^\frac{3}{2}}{\alpha_n} = \sqrt{\alpha_n}.\]
Hence linear stabilisation gives
\[\left\Vert e^{(T_n - k) A} \int_0^1 e^{- s A} F_n(k + s) \mathrm{d}s \right\Vert_{H_0^1 \times L^2} \lesssim e^{\lambda (k - T_n)} \sqrt{\alpha_n} \]
where $\lambda > 0$ is a constant, implying
\[\left\Vert \sum_{k = 0}^{\lfloor T_n \rfloor - 1} e^{(T_n - k) A} \int_0^1 e^{- s A} F_n(k + s) \mathrm{d}s \right\Vert_{H_0^1 \times L^2} \lesssim \sum_{k = 0}^{\lfloor T_n \rfloor - 1} e^{\lambda (k - T_n)} \sqrt{\alpha_n} \lesssim \sqrt{\alpha_n}.\]
Similarly, one has
\[\left\Vert \int_{\lfloor T_n \rfloor}^{T_n} e^{(T_n - s) A} F_n(s) \mathrm{d}s \right\Vert_{H_0^1 \times L^2} \lesssim \sqrt{\alpha_n}.\]
As $\alpha_{\phi(n)} \rightarrow 0$, coming back to (\ref{eq_proof_thm_stab_GCC_3}), one finds 
\[W_{\phi(n)}(T_{\phi(n)}) \xrightarrow{n \rightarrow \infty} 0\]
in $H_0^1(\Omega) \times L^2(\Omega)$, a contradiction with (\ref{eq_proof_thm_stab_GCC_2}). This completes the proof of Theorem \ref{thm_stab_GCC_introduction}.

\printbibliography

\noindent
\textsc{Perrin Thomas:} \texttt{perrin@math.univ-paris13.fr}

\noindent
\textit{Laboratoire Analyse Géométrie et Application, Institut Galilée - UMR 7539, CNRS/Université Sorbonne Paris Nord, 99 avenue J.B. Clément, 93430 Villetaneuse, France}

\end{document}